\documentclass[11pt,a4paper]{article}
\usepackage{authblk}
\usepackage{bm}
\usepackage{geometry}
\usepackage{fancyhdr}
\usepackage{tikz}
\usetikzlibrary{shapes.geometric}
\usepackage{booktabs}
\usepackage{amsmath,amsthm,amssymb,amsfonts,mathrsfs,amscd,mathtools}
\usepackage{enumerate}
\usepackage[hidelinks,unicode]{hyperref}
\allowdisplaybreaks[2]
\newtheorem{thm}{Theorem}[section]
\newtheorem{prop}[thm]{Proposition}
\newtheorem{lem}[thm]{Lemma}
\newtheorem{cor}[thm]{Corollary}

\newenvironment{abs}{\par\noindent\textbf{Abstract.}\enspace\normalfont\ignorespaces}{\par}
\theoremstyle{definition}
\newtheorem{definition}[thm]{Definition}
\newtheorem{example}[thm]{Example}

\newtheorem{remark}[thm]{Remark}
\numberwithin{equation}{section}
\renewcommand{\labelenumi}{\textup{(\arabic{enumi})}}
\newcommand{\D}{\mathbb D}
\newcommand{\T}{\mathbb T}
\newcommand{\Cseq}{\mathcal C}
\newcommand{\CM}{\mathrm{CM}}
\newcommand{\Cplus}{\mathcal C_{+}^{\mathrm c}}
\newcommand{\Cbi}{\mathcal C_{\leftrightarrow}^{\mathrm c}}
\newcommand{\Pc}{\mathcal P_{\mathrm c}}
\newcommand{\Kc}{\mathcal K_{\mathrm c}}

\hypersetup{pdftitle={Characterizations and boundary regularity of continuous preservers of Carleson interpolating sequences},pdfauthor={Jian Wu}}

\begin{document}

\title{Characterizations and boundary regularity of continuous preservers of Carleson interpolating sequences}
\author{Jian Wu\thanks{Corresponding author: \texttt{xingxingwu2022@163.com}}}
\affil{}
\renewcommand*{\Affilfont}{\small\itshape}
\renewcommand\Authands{ and }
\date{}
\maketitle

\begin{abs}
We characterize the continuous mappings of the unit disk that preserve every
Carleson interpolating sequence. Preservation in one direction is equivalent to
being a disk homeomorphism whose inverse is uniformly continuous in the
pseudohyperbolic metric and whose associated weighted pushforward is bounded
on positive Carleson measures. Preservation in both directions is equivalent
to being a disk homeomorphism that, together with its inverse, is uniformly
continuous in that metric and has a strongly quasisymmetric boundary extension. Within this class,
comparability of the boundary defects is equivalent to a bi-Lipschitz
boundary map. Building on the canonical factorization established in our
earlier work, we give explicit criteria for membership in the continuous
boundary kernel. For a homeomorphism of the closed disk
with identity boundary values and a uniform boundary expansion of first order,
kernel membership is equivalent to positivity of the determinant of the real boundary differential,
or equivalently of the inward normal derivative of the boundary defect.
No differentiability in the interior is required. The $C^1$ criterion on the
closed disk follows as a corollary, while a counterexample shows that
pointwise boundary differentiability does not suffice. We also characterize
factorizations with a quasiconformal kernel and record the
exact maximal dilatation of its radial factor.
\end{abs}



\section{Introduction}
\label{sec:introduction}

Carleson interpolating sequences provide a basic link between bounded
analytic interpolation, reproducing kernels, and the geometry of the unit
disk. Write $\D=\{z\in\mathbb C:|z|<1\}$ and $\T=\partial\D$.
A sequence $\Lambda=(\lambda_j)$ in $\D$ is interpolating for
$H^\infty(\D)$ if every bounded scalar sequence is the restriction to
$\Lambda$ of a function in $H^\infty(\D)$. We call these sequences
\emph{Carleson sequences} and denote their class by $\Cseq$.
Carleson's interpolation theorem \cite{Carleson1958} characterizes them by
\(
 \inf_j\prod_{k\ne j}\rho(\lambda_j,\lambda_k)>0,
 \rho(z,w)=\left|\frac{z-w}{1-\overline z w}\right|.
\)
Equivalently, $\Lambda$ is pairwise separated in the pseudohyperbolic
metric $\rho$, and its canonical measure
\(
 \mu_\Lambda=\sum_j(1-|\lambda_j|^2)\delta_{\lambda_j}
\)
is a Carleson measure; see \cite{Garnett2007}. Throughout this paper,
sequences are indexed families, their images retain repeated terms, and
finite sequences of distinct points are included in $\Cseq$.

The same class has a Hilbert space interpretation. Let
\(
 k_\lambda(z)=\frac{1}{1-\overline\lambda z},
 \widehat k_\lambda=(1-|\lambda|^2)^{1/2}k_\lambda
\)
be the Szeg\H{o} kernel of $H^2(\D)$ and its normalization.
The theorem of Shapiro and Shields \cite{ShapiroShields1961} gives
\(
 \Lambda\in\Cseq
 \Longleftrightarrow
 (\widehat k_{\lambda_j})_j
 \text{ is a Riesz sequence in }H^2(\D);
\)
see also \cite{Garnett2007}. This connection extends to normalized
complete Pick spaces: Aleman, Hartz, McCarthy and Richter
\cite{AlemanHartzMcCarthyRichter2019} characterized interpolating
sequences for the multiplier algebra by weak separation and the associated
Carleson measure condition. More recently, Hartz
\cite{Hartz2023} proved the column-row property
with constant one and obtained a new proof of this characterization.
These results place the classical Szeg\H{o} kernel formulation in a
broader setting in interpolation theory.

The characterization by separation and a canonical Carleson measure leads
to two closely related preservation problems: which mappings preserve
Carleson measures under the appropriate transport, and which mappings
preserve Carleson sequences? The former concerns measure
estimates, while the latter also involves the separation of the points
supporting the canonical measure.

For Carleson measures, Semmes \cite{Semmes1988} obtained weighted
pullback estimates for quasiconformal disk homeomorphisms that are
bi-Lipschitz in the Poincar\'e metric and have strongly quasisymmetric
boundary values. Zinsmeister \cite{Zinsmeister1989} related boundedness of
weighted pullback and pushforward operators induced by conformal maps to
$\mathrm{BMOA}$ and boundary geometry. Wei and Zinsmeister
\cite{WeiZinsmeister2018} developed this theory further, including
vanishing Carleson measures and applications to chord-arc domains.
Gonz\'alez and Nicolau \cite[Corollary~1]{GN1998} also characterized
weighted Carleson measure preservation by quasiconformal homeomorphisms of the
upper half plane in terms of strong quasisymmetry of the boundary map. For arbitrary disk bijections, M\"uller
\cite[Theorem~4]{Muller1997} characterized uniform two-sided
Carleson norm comparison between the canonical atomic measures
of a sequence and its image by geometric decomposition conditions
on the map and its inverse.
These results emphasize that the relevant transport generally includes a weight reflecting the change of scale near the boundary.

For Carleson sequences, a complete description is known
within the quasiconformal category. Astala and Zinsmeister
\cite{AZ1991} established preservation when the boundary map is strongly
quasisymmetric, and Gonz\'alez and Nicolau \cite{GN1998} proved the
converse. Their formulation on the upper half plane yields the corresponding
disk statement by conformal equivalence. Strong quasisymmetry means
absolute continuity with an $A_\infty$ density, and strongly
quasisymmetric circle homeomorphisms form a group. Applying the
characterization to the inverse therefore gives preservation in both
directions. This raises the question of what remains true when
quasiconformality is replaced by continuity alone.

We study continuous maps $\Phi:\D\to\D$ satisfying, for every indexed
sequence $\Lambda$, either
\(
 \Lambda\in\Cseq \Longrightarrow \Phi(\Lambda)\in\Cseq\) or
 \(\Lambda\in\Cseq \Longleftrightarrow \Phi(\Lambda)\in\Cseq.
\)
The corresponding classes are denoted by $\Cplus$ and $\Cbi$,
respectively. Neither injectivity nor surjectivity is assumed in their
definitions. Our first aim is to characterize these classes by separating
the metric requirement from the measure requirement, and then to identify
the boundary geometry of bidirectional preservation. A disk homeomorphism
whose forward and inverse maps are uniformly continuous in $\rho$ will be
called a \emph{$\rho$-uniform homeomorphism}.
We use \emph{identity boundary values} to mean that a map, or its
continuous extension to $\overline\D$, restricts to
$\operatorname{id}_{\T}$ on $\T$.

The canonical measures determine the transport appropriate to this
problem. For a continuous map $\Phi:\D\to\D$, set
\(
 q_\Phi(z)=\frac{1-|\Phi(z)|^2}{1-|z|^2},
 T_\Phi\mu=\Phi_*(q_\Phi\mu).
\)
Here $F_*\sigma$ is the ordinary pushforward of a positive Borel measure
$\sigma$ under a Borel map $F:X\to Y$, defined by
$(F_*\sigma)(E)=\sigma(F^{-1}(E))$ for Borel sets $E\subset Y$.
The weight gives the exact identity
\(
 T_\Phi\mu_\Lambda=\mu_{\Phi(\Lambda)},
\)
with multiplicities counted. Theorem~\ref{thm:transport-one-sided}
shows that $\Phi\in\Cplus$ if and only if $\Phi$ is a disk
homeomorphism, $\Phi^{-1}$ is uniformly continuous in $\rho$, and
\(
 \|T_\Phi\mu\|_{\CM}\leq C_\Phi\|\mu\|_{\CM}\) for \(\mu\in\CM\).
The constant is independent of the positive Carleson measure $\mu$.
The proof first obtains uniform estimates for discrete measures with
controlled separation and Carleson constants. Discretization by
interpolating sequences then gives qualitative preservation of all
positive Carleson measures, and a summation argument on the positive cone
supplies the uniform bound. Applying the result to both $\Phi$ and
$\Phi^{-1}$ yields the symmetric weighted transport characterization in
Theorem~\ref{thm:transport-two-sided}. The two classes are distinct:
the smooth map $F(z)=2z/(1+|z|^2)$ has identity boundary values and
belongs to $\Cplus\setminus\Cbi$; see
Example~\ref{ex:one-sided-not-two-sided}. Preservation in one direction need not even provide radial
boundary values: Example~\ref{ex:one-sided-no-radial-limits}
gives a smooth disk diffeomorphism with no radial limit at
any point of $\T$.
On the other hand,
Corollary~\ref{cor:automatic-bidirectional-preservation}
shows that a map in $\Cplus$ belongs to $\Cbi$ exactly when
it is uniformly continuous with respect to $\rho$.

The boundary characterization is
\[
 \Phi\in\Cbi
 \quad\Longleftrightarrow\quad
 \begin{gathered}
 \Phi\text{ is a $\rho$-uniform disk homeomorphism},\\
 h_\Phi\in\operatorname{SQS}^{\pm}(\T),
 \end{gathered}
\]
where $h_\Phi$ is the continuous boundary extension and
$\operatorname{SQS}^{\pm}(\T)$ includes both orientations; see
Theorem~\ref{thm:continuous-bidirectional-boundary}.
A $\rho$-uniform homeomorphism is a quasi-isometry for the hyperbolic
metric. Its boundary map is consequently quasisymmetric, and identity
boundary values force bounded hyperbolic displacement. These facts allow
us to compare a continuous preserver with a quasiconformal extension of
its boundary map and apply the theorem of Gonz\'alez and Nicolau. The classical quasiconformal characterization is recovered in
Corollary~\ref{cor:classical-qc-specialization}.
Example~\ref{ex:non-qc-bidirectional-preserver} shows that our
classification strictly extends the quasiconformal setting,
even among $C^1$ homeomorphisms of the closed disk.

Bidirectional preservation of interpolation does not, however, determine
the scale of the boundary defect $1-|z|^2$, even within the
quasiconformal class; see Example~\ref{ex:power-stretch}. This leads to
the subclass
\(
 \Pc=\{\Phi\in\Cbi:q_\Phi\asymp1\},
 q_\Phi\asymp1
 \Longleftrightarrow
 1-|\Phi(z)|^2\asymp1-|z|^2\) for \(z\in\D\),
where the comparison constants are independent of $z$. From the
kernel viewpoint, the additional condition follows naturally from
\(
 \|k_z\|_{H^2}^2=(1-|z|^2)^{-1}.
\)
Thus $\Cbi$ preserves the Riesz sequence property of normalized
Szeg\H{o} kernels, whereas $\Pc$ also preserves the norms of the
unnormalized kernels up to uniform factors.

The condition on boundary defects is also intrinsic to the preservation of
frames generated by operator orbits. For a diagonal normal operator
$Ae_j=\lambda_j e_j$
with respect to an orthonormal basis $(e_j)_{j\geq1}$, where
$\lambda_j\in\D$, and $f=\sum_j f_je_j$, the characterization of frames generated by a single orbit
\cite[Theorem~5.7]{IterativeNormalOperators}
(see also \cite[Theorem~5.1]{OperatorOrbitFrames}) states that
$(A^nf)_{n\geq0}$ is a frame if and only if $(\lambda_j)\in\Cseq$
and $|f_j|^2\asymp1-|\lambda_j|^2$. Thus preserving the same frame
generators under $A\mapsto\Phi(A)$ requires compatibility of these
weights. Requiring equality of the generator sets for every such diagonal
operator forces the uniform comparison $q_\Phi\asymp1$; see
\cite[Theorem~2.8]{Wu2026}.

This paper characterizes continuous preservers of Carleson
sequences without assuming comparability of boundary defects.
For the subclass $\Pc$ with comparable defects, the bi-Lipschitz boundary
conclusion and canonical factorization follow by specializing the
corresponding results of \cite[Theorems~4.4 and~5.2, Corollary~5.3]{Wu2026} to
continuous maps and combining them with our boundary characterization.
More precisely, Theorem~\ref{thm:weighted-continuous-characterization}
shows that, among $\rho$-uniform disk homeomorphisms, comparability of
boundary defects is equivalent to a bi-Lipschitz boundary map, and either
condition characterizes $\Pc$. Every $\Phi\in\Pc$ has the unique
canonical factorization
\(
 \Phi=\Psi\circ E_{h_\Phi},
 E_h(0)=0, E_h(r\zeta)=rh(\zeta),
\)
where $h_\Phi\in\operatorname{BiLip}(\T)$ and
\(
 \Psi=\Phi\circ E_{h_\Phi^{-1}}\in
 \Kc:=\{F\in\Pc:h_F=\operatorname{id}_{\T}\}.
\)
Theorem~\ref{thm:explicit-continuous-kernel} identifies $\Kc$ with
the $\rho$-uniform disk homeomorphisms with identity boundary values,
and gives equivalent criteria in terms of displacement and pairs of points.
In the continuous setting, this factorization gives the
group decomposition
\(
 \Pc\cong\Kc\rtimes\operatorname{BiLip}(\T);
\)
see Corollary~\ref{cor:continuous-semidirect-product}. This separates the
boundary action from the factor with identity boundary values and provides a
common framework for the differentiable and quasiconformal classes.

For boundary differentiability, let $\Psi$ be a homeomorphism of the
closed disk with identity boundary values, and suppose that, for a continuous
$v_\Psi:\T\to\mathbb C$,
\(
 \Psi((1-t)\zeta)=\zeta-t v_\Psi(\zeta)+o(t)\) for \(t\downarrow0\)
uniformly in $\zeta\in\T$. Theorem~\ref{thm:uniform-boundary-normal}
then gives
\[
 \Psi|_{\D}\in\Kc
 \quad\Longleftrightarrow\quad
 \inf_{\zeta\in\T}
 \operatorname{Re}\bigl(\overline\zeta v_\Psi(\zeta)\bigr)>0.
\]
The expansion determines the real differential relative to the closed
disk at every boundary point. At $\zeta$, its determinant is
$\operatorname{Re}(\overline\zeta v_\Psi(\zeta))$, which is also
the inward normal derivative of the boundary defect
$1-|\Psi(z)|$. No differentiability is required in the interior.
Uniformity is essential: in
Example~\ref{ex:pointwise-boundary-insufficient}, every boundary
differential is the identity, yet the map does not belong to $\Kc$.
If $\Psi\in C^1(\overline\D)$, then
$v_\Psi(\zeta)=D\Psi(\zeta)[\zeta]$, yielding the $C^1$ criterion on the
closed disk in Corollary~\ref{cor:c1-boundary-normal}.
The differential may still degenerate in the interior, as illustrated by
$\Psi(z)=|z|^2z$. Regularity of the canonical kernel must also be
distinguished from regularity of the full map: even for smooth $h$, the
radial extension $E_h$ need not be differentiable at the origin.

For quasiconformal regularity, every quasiconformal disk homeomorphism
$\Psi$ with identity boundary values belongs to $\Kc$.
For such a fixed kernel and a circle homeomorphism $h$ that preserves
orientation, Theorem~\ref{thm:qc-kernel-classification} gives
\[
 \Psi\circ E_h\in\Pc
 \quad\Longleftrightarrow\quad
 h\in\operatorname{BiLip}^{+}(\T)
 \quad\Longleftrightarrow\quad
 \Psi\circ E_h\text{ is quasiconformal}.
\]
We also record the exact maximal dilatation of the radial
factor and the resulting composition bound; see
\cite{Kalaj2014,Ahlfors2006} and
Section~\ref{subsec:qc-kernel}.
The statements for maps that reverse orientation follow by composition with complex
conjugation. In particular, for a quasiconformal or antiquasiconformal
disk homeomorphism, membership in $\Pc$, a bi-Lipschitz boundary map,
and comparability of boundary defects are equivalent; see
Corollary~\ref{cor:all-qc-boundary-defect}.

Section~\ref{sec:continuous-preservers} develops the transport and
boundary characterizations and culminates in
Theorem~\ref{thm:continuous-bidirectional-boundary}.
Section~\ref{sec:regularity} establishes the common factorization
framework, treats uniform boundary expansions of first order in
Section~\ref{subsec:boundary-first-order}, and studies quasiconformal
mappings in Section~\ref{subsec:qc-kernel}.

We conclude with the measure conventions used below. For an arc
$I\subset\T$, let $|I|$ be its normalized arclength, and let $d_\T$
be the corresponding distance measured along the shorter arc, with values in $[0,1/2]$.
For $0<|I|<1$, set
\(
 S(I)=\{r\zeta:\zeta\in I,\ 0<r<1,\ 1-r\leq|I|\}, S(\T)=\D.
\)
We may use closed arcs; the choice of endpoints does not affect the
equivalent norm estimates. For a positive Borel measure $\mu$ on $\D$,
define
\(
 \|\mu\|_{\CM}=
 \sup_{0<|I|\leq1}\frac{\mu(S(I))}{|I|},
 \CM=\{\mu\geq0:\|\mu\|_{\CM}<\infty\}.
\)
The test $I=\T$ controls total mass and includes atoms at the origin.
We write $\delta(z)=1-|z|$ and
$\operatorname{sep}(\Lambda)=\inf_{j\ne k}\rho(\lambda_j,\lambda_k)$.
The notation $A\lesssim B$ means $A\leq CB$ with a constant
independent of the points, arcs, or sequence indices under consideration,
and $A\asymp B$ means that both comparisons hold. Constants may depend
on a fixed mapping.

\section{Continuous preservers of Carleson sequences}
\label{sec:continuous-preservers}

Throughout this section, sequences are indexed families, and their images retain
repetitions. A finite sequence of distinct points is regarded as interpolating.
For a set $E\subset\D$ of distinct points, write
\[
 \mu_E=\sum_{z\in E}(1-|z|^2)\delta_z,
 \qquad \delta(z)=1-|z|.
\]
We use the following form of Carleson's interpolation theorem: $E\in\Cseq$ if
and only if $E$ is pairwise separated in the pseudohyperbolic metric and
$\mu_E\in\CM$; see \cite{Carleson1958,Garnett2007}. Here pairwise separation
means that the distances between distinct points have a common positive lower
bound, as distinguished from the product separation condition.

\subsection{Weighted transport characterizations}
\label{subsec:weighted-transport}
We first record the standard kernel tests for Carleson measures;
see \cite{Garnett2007}. Write
\(
 \mathbb H=\{w\in\mathbb C:\operatorname{Im}w>0\}.
\)
For a bounded interval $J\subset\mathbb R$ of positive Euclidean
length $|J|$, set
\(
 Q(J)=\{x+iy:x\in J,\ 0<y\leq |J|\}.
\)
For a positive Borel measure $\sigma$ on $\mathbb H$, define
\(
 \|\sigma\|_{\CM(\mathbb H)}
 =\sup_J\frac{\sigma(Q(J))}{|J|}.
\)

\begin{lem}
\label{lem:transport-kernel-tests}
For positive locally finite Borel measures $\mu$ on $\D$ and
$\sigma$ on $\mathbb H$, one has
\(
 \|\mu\|_{\CM}
 \asymp
 \sup_{a\in\D}
 \int_\D
 \frac{1-|a|^2}{|1-\overline a z|^2}\,d\mu(z)
\)
and
\(
 \|\sigma\|_{\CM(\mathbb H)}
 \asymp
 \sup_{b\in\mathbb H}
 \int_{\mathbb H}
 \frac{\operatorname{Im}b}{|w-\overline b|^2}\,d\sigma(w).
\)
The implicit constants in both comparisons are absolute
and independent of the measures $\mu$ and $\sigma$.
\end{lem}

For $U=\D$ or $\mathbb H$ and a nonempty sequence
$\Lambda=(\lambda_j)\subset U$, let
$\mathfrak I_U(\Lambda)$ denote its $H^\infty(U)$ interpolation
constant. This is the infimum of all $B\geq1$ such that every
bounded sequence $(a_j)$ admits a function $f\in H^\infty(U)$
satisfying
\(
 f(\lambda_j)=a_j\) for all \(j\),
 \(\|f\|_{H^\infty(U)}\leq B\|(a_j)\|_{\ell^\infty}.
\)
We set $\mathfrak I_U(\Lambda)=+\infty$ if no such $B$ exists.

\begin{lem}
\label{lem:transport-interpolation-bounds}
There is an absolute constant $C$ such that, for every nonempty
sequence $\Lambda=(\lambda_j)$ of distinct points in $\D$,
\(
 \mathfrak I_\D(\Lambda)\leq L<\infty
 \Longrightarrow
 \operatorname{sep}(\Lambda)\geq L^{-1},
 \|\mu_\Lambda\|_{\CM}\leq CL^2.
\)
The separation assertion is vacuous for a singleton.
\end{lem}

\begin{proof}
Fix $R>L$. For distinct indices $j$ and $k$, choose
$f\in H^\infty(\D)$ with
\(
 f(\lambda_j)=1, f(\lambda_k)=0,
  \|f\|_\infty\leq R.
\)
Applying the Schwarz--Pick lemma to $f/R$ gives
\(
 R^{-1}\leq \rho(\lambda_j,\lambda_k).
\)
Since $R>L$ is arbitrary, this proves the separation estimate.

To prove the measure estimate, put
\(
 e_j(z)=\frac{(1-|\lambda_j|^2)^{1/2}}
                   {1-\overline{\lambda_j}z}.
\)
These are the normalized reproducing kernels of \(H^2(\D)\).
Let \(F\) be a finite set of indices and let \((c_j)_{j\in F}\)
be arbitrary complex numbers. For each choice of signs
\(\varepsilon=(\varepsilon_j)_{j\in F}\in\{-1,1\}^F\),
choose \(f_\varepsilon\in H^\infty(\D)\) such that
\(
 f_\varepsilon(\lambda_j)=\varepsilon_j
 \quad(j\in F),
 f_\varepsilon(\lambda_j)=0
 \quad(j\notin F),
 \|f_\varepsilon\|_\infty\leq R.
\)
Let $M_{f_\varepsilon}$ denote multiplication by
$f_\varepsilon$ on $H^2(\D)$. Since
$\|M_{f_\varepsilon}\|\leq R$ and
\(
 M_{f_\varepsilon}^*e_j
 =\overline{f_\varepsilon(\lambda_j)}e_j,
\)
we obtain
\(
 \left\|\sum_{j\in F}c_je_j\right\|_{H^2}^2
 \leq
 R^2\left\|\sum_{j\in F}\varepsilon_jc_je_j
          \right\|_{H^2}^2.
\)
Averaging over all sign choices, and using $\|e_j\|_{H^2}=1$,
gives
\(
 \left\|\sum_{j\in F}c_je_j\right\|_{H^2}^2
 \leq R^2\sum_{j\in F}|c_j|^2.
\)
By duality, and since $F$ is arbitrary,
\(
 \sum_j(1-|\lambda_j|^2)|h(\lambda_j)|^2
 \leq R^2\|h\|_{H^2}^2\) for \(h\in H^2(\D)\).
Taking
\(
 h(z)=\frac{(1-|a|^2)^{1/2}}{1-\overline a z}\) for \(a\in\D,\)
whose $H^2$ norm is one, yields
\[
 \sup_{a\in\D}
 \int_\D
 \frac{1-|a|^2}{|1-\overline a z|^2}\,d\mu_\Lambda(z)
 \leq R^2.
\]
By Lemma~\ref{lem:transport-kernel-tests},
$\|\mu_\Lambda\|_{\CM}\leq CR^2$.
Letting $R\downarrow L$ proves the assertion.
\end{proof}

\begin{lem}
\label{lem:transport-finite-perturbation}
An interpolating sequence has only finitely many terms in each compact subdisk
of $\D$. Adding or deleting finitely many points does not change its
interpolating property, provided that no repetitions are introduced.
\end{lem}

These standard properties follow from the characterization by separation and
the Carleson measure condition
above; see also \cite{ShapiroShields1961}.

\begin{samepage}
\begin{lem}
\label{lem:transport-sparse-packets}
Let $F_n\subset\D$ be nonempty finite sets with $\#F_n\leq2$. Suppose that
distinct points in each $F_n$ have pseudohyperbolic distance at least
$\varepsilon>0$. Set
\(
 b_n=\min_{z\in F_n}\delta(z),
 a_n=\max_{z\in F_n}\delta(z).
\)
If $a_{n+1}\leq\theta b_n$ for all $n$, where $0<\theta<1$, then
$\bigcup_nF_n\in\Cseq$.
\end{lem}
\end{samepage}

\begin{proof}
Set $E=\bigcup_n F_n$. We first verify separation.
If $z\in F_n$ and $w\in F_m$ with $m>n$, then
\(
 \delta(w)\leq a_m\leq\theta^{m-n}b_n
 \leq\theta\delta(z).
\)
Writing $x=\delta(z)$ and $y=\delta(w)$, we obtain
\[
 \rho(z,w)\geq
 \frac{\bigl||z|-|w|\bigr|}{1-|z||w|}
 =\frac{x-y}{x+y-xy}
 \geq\frac{x-y}{x+y}
 \geq\frac{1-\theta}{1+\theta}.
\]
Thus different packets are disjoint, and the separation assumption within each
packet implies that $E$ is pairwise separated.

We next show that $\mu_E$ is a Carleson measure.
Fix an arc $I\subset\T$ and put $t=|I|$.
If no index $n$ satisfies $b_n\leq t$, then
$E\cap S(I)=\varnothing$.
Otherwise, let $n_0$ be the first such index.
Packets with $n<n_0$ do not meet $S(I)$.
Since $\#F_{n_0}\leq2$ and
$1-|z|^2\leq2\delta(z)$, we have
\[
 \mu_{F_{n_0}}(S(I))
 \leq 2\sum_{z\in F_{n_0}\cap S(I)}\delta(z)
 \leq4t.
\]
For every $k\geq1$, the hypothesis gives
\(
 a_{n_0+k}\leq\theta^k b_{n_0}\leq\theta^k t.
\)
Consequently,
\[
 \mu_E(S(I))
 \leq \mu_{F_{n_0}}(S(I))
       +\sum_{k=1}^{\infty}\mu_{F_{n_0+k}}(\D)
 \leq 4t+4\sum_{k=1}^{\infty}a_{n_0+k}
 \leq 4t\sum_{k=0}^{\infty}\theta^k
 =\frac{4t}{1-\theta}.
\]
Hence $\mu_E\in\CM$, and Carleson's interpolation theorem
yields $E\in\Cseq$.
For singleton packets, separation follows entirely from
the estimate between different packets.
\end{proof}

\begin{samepage}
\begin{prop}
\label{prop:transport-injective-proper}
Let $\Phi:\D\to\D$ be an arbitrary map such that
\(
 \Lambda\in\Cseq\Longrightarrow\Phi(\Lambda)\in\Cseq\) for every indexed sequence \(\Lambda\subset\D
\).
Then $\Phi$ is injective, and
\(
 |z_n|\rightarrow1
 \Longrightarrow
 |\Phi(z_n)|\rightarrow1.
\)
\end{prop}
\end{samepage}

\begin{proof}
Every sequence consisting of two distinct points is interpolating, so its image
cannot contain a repetition. Thus $\Phi$ is injective.

If the asserted boundary behavior failed, there would exist
$R\in(0,1)$ and a sequence $(z_n)\subset\D$ such that
$|z_n|\to1$ and $|\Phi(z_n)|\leq R$ for every $n$. Passing to a subsequence, arrange that
$\delta(z_{n+1})\leq\delta(z_n)/4$. Lemma~\ref{lem:transport-sparse-packets}
shows that this subsequence is interpolating, whereas its image contains
infinitely many distinct points in a compact subdisk, contrary to
Lemma~\ref{lem:transport-finite-perturbation}.
\end{proof}

The next proposition supplies the uniform discrete estimates
needed to pass from sequence preservation to bounded transport
of Carleson measures.

\begin{samepage}
\begin{prop}
\label{prop:transport-uniformization}
Suppose that $\Phi:\D\to\D$ maps every interpolating sequence to an
interpolating sequence. For every $\eta>0$ and $M<\infty$, there is a constant
$B(\Phi,\eta,M)<\infty$ such that
\[
 \operatorname{sep}(\Lambda)\geq\eta,
 \qquad \|\mu_\Lambda\|_{\CM}\leq M
 \quad\Longrightarrow\quad
 \|\mu_{\Phi(\Lambda)}\|_{\CM}\leq B(\Phi,\eta,M).
\]
\end{prop}
\end{samepage}

\begin{proof}
By Proposition~\ref{prop:transport-injective-proper}, $\Phi$ is injective.
Suppose that the assertion fails for fixed $\eta$ and $M$.

We first claim that, for every $0<s<1/2$ and $L>0$, there exist
a sequence $\Lambda$ satisfying
\(
 \operatorname{sep}(\Lambda)\geq\eta, \|\mu_\Lambda\|_{\CM}\leq M,
\)
arcs $I,J\subset\T$, and a nonempty finite set
\[
 F\subset\Lambda\cap S(J)\cap\Phi^{-1}(S(I)),
 \qquad |J|=s,
\]
such that
\(
 \mu_{\Phi(F)}(S(I))>L|I|.
\)

To prove the claim, fix $0<s<1/2$ and $L>0$.
Choose arcs $J_1,\ldots,J_{N_s}$ of length $s$ covering $\T$,
where $N_s\leq2/s$.
By the assumed failure of the uniform bound, there exist
a sequence $\Lambda$ satisfying the separation and Carleson bounds
above and an arc $I\subset\T$ such that
\[
 \mu_{\Phi(\Lambda)}(S(I))
 >
 \left(N_sL+\frac{2M}{s}\right)|I|.
\]
Let
\(
 H=\{z\in\Lambda:\delta(z)>s\}.
\)
Since $1-|z|^2\geq\delta(z)>s$ for $z\in H$ and
$\mu_\Lambda(\D)\leq M$, the set $H$ has at most $M/s$ points.
Moreover, if $\Phi(z)\in S(I)$, then
\(
 1-|\Phi(z)|^2
 \leq2\bigl(1-|\Phi(z)|\bigr)
 \leq2|I|.
\)
Consequently,
\[
 \sum_{\substack{z\in H\\ \Phi(z)\in S(I)}}
       (1-|\Phi(z)|^2)
 \leq\frac{2M}{s}|I|.
\]
Thus the remaining set
\(
 E=\{z\in\Lambda:\delta(z)\leq s,\ \Phi(z)\in S(I)\}
\)
satisfies
\(
 \sum_{z\in E}(1-|\Phi(z)|^2)>N_sL|I|.
\)

Every point of $E$ is nonzero and belongs to at least one of the
boxes $S(J_1),\ldots,S(J_{N_s})$. Hence
\[
 \sum_{j=1}^{N_s}
 \sum_{z\in E\cap S(J_j)}(1-|\Phi(z)|^2)
 \geq
 \sum_{z\in E}(1-|\Phi(z)|^2)
 >N_sL|I|.
\]
It follows that, for some $j$,
\(
 \sum_{z\in E\cap S(J_j)}(1-|\Phi(z)|^2)>L|I|.
\)
Since all summands are positive, there is a nonempty finite set
$F\subset E\cap S(J_j)$ such that
\(
 \mu_{\Phi(F)}(S(I))
 =\sum_{z\in F}(1-|\Phi(z)|^2)>L|I|.
\)
Taking $J=J_j$ proves the claim.

Fix $0<\theta<1/4$.
Apply the claim inductively, with $L=n$ at the $n$th step,
to obtain nonempty finite sets $F_n\subset S(J_n)$ and
arcs $I_n\subset\T$.
Writing
\(
 s_n=|J_n|,
 b_n=\min_{z\in F_n}\delta(z),
\)
choose the scales so that
\(
 \mu_{\Phi(F_n)}(S(I_n))>n|I_n|,
 s_{n+1}<\theta\min\{s_n,b_n\}.
\)
Each $F_n$ is contained in a sequence satisfying the prescribed
separation and Carleson bounds. It therefore inherits separation
at least $\eta$ and Carleson norm at most $M$.
Moreover,
\(
 \mu_{F_n}(\D)
 =\mu_{F_n}(S(J_n))
 \leq Ms_n.
\)

Let $\Gamma=\bigcup_n F_n$.
If $z\in F_n$ and $w\in F_m$ with $m>n$, then
\(
 \delta(w)\leq s_m
 \leq s_{n+1}
 <\theta b_n
 \leq\theta\delta(z).
\)
The radial estimate in the proof of
Lemma~\ref{lem:transport-sparse-packets} gives
\(
 \rho(z,w)\geq\frac{1-\theta}{1+\theta}.
\)
Thus the sets $F_n$ are pairwise disjoint, and $\Gamma$ is
pairwise separated.

To estimate its canonical measure, fix an arc $I\subset\T$
and put $t=|I|$.
If no index $n$ satisfies $b_n\leq t$, then
$\Gamma\cap S(I)=\varnothing$.
Otherwise, let $n_0$ be the first such index.
The sets $F_n$ with $n<n_0$ do not meet $S(I)$, while
the inherited Carleson bound gives
\(
 \mu_{F_{n_0}}(S(I))\leq Mt.
\)
Since
\(
 s_{n_0+1}<\theta b_{n_0}\leq\theta t,
  s_{n+1}<\theta s_n,
\)
we have $s_{n_0+k}\leq\theta^k t$ for every $k\geq1$.
Consequently,
\[
 \sum_{n>n_0}\mu_{F_n}(S(I))
 \leq\sum_{n>n_0}\mu_{F_n}(\D)
 \leq M\sum_{k=1}^{\infty}s_{n_0+k}
 \leq Mt\sum_{k=1}^{\infty}\theta^k
 =\frac{M\theta}{1-\theta}\,t.
\]
It follows that
\[
 \mu_\Gamma(S(I))
 \leq Mt+\frac{M\theta}{1-\theta}\,t
 =\frac{M}{1-\theta}|I|.
\]
Hence $\mu_\Gamma$ is a Carleson measure.
Together with pairwise separation, Carleson's interpolation
theorem yields $\Gamma\in\Cseq$.

On the other hand, $F_n\subset\Gamma$ and injectivity of $\Phi$
give, for every $n$,
\[
 \frac{\mu_{\Phi(\Gamma)}(S(I_n))}{|I_n|}
 \geq
 \frac{\mu_{\Phi(F_n)}(S(I_n))}{|I_n|}
 >n.
\]
Thus $\mu_{\Phi(\Gamma)}$ is not a Carleson measure, contradicting
the assumption that $\Phi$ maps every interpolating sequence
to an interpolating sequence.
\end{proof}

To pass from canonical measures of interpolating sequences
to arbitrary positive Carleson measures, we use the following
form of Garnett's discretization.

\begin{samepage}
\begin{lem}
\label{lem:transport-garnett-discretization}
There is an absolute constant $A_0<\infty$ with the following property.
For every positive Carleson measure $\mu$ with $\|\mu\|_{\CM}\leq1$, there
exist $\eta_\mu>0$, $M_\mu<\infty$, and positive measures
\(
 \mu_n=\sum_{k=1}^{N_n}c_{n,k}\mu_{\Lambda_{n,k}},
  c_{n,k}\geq0,
 \sum_{k=1}^{N_n}c_{n,k}\leq A_0,
\)
such that $\mu_n\to\mu$ vaguely and
\(
 \operatorname{sep}(\Lambda_{n,k})\geq\eta_\mu,
  \|\mu_{\Lambda_{n,k}}\|_{\CM}\leq M_\mu\) for all \(n,k\).
\end{lem}
\end{samepage}

\begin{proof}
The assertion is immediate when $\mu=0$, by taking all
coefficients equal to zero. Assume henceforth that $\mu\ne0$.

We first transfer $\mu$ to the upper half plane. Set
\(
 \tau(z)=i\frac{1+z}{1-z},
 \nu=\tau_*\bigl(|\tau'|\mu/2\bigr).
\)
For $a,z\in\D$, the identities
\(
 |\tau'(z)|=\frac{2}{|1-z|^2},
 \operatorname{Im}\tau(a)
 =\frac{1-|a|^2}{|1-a|^2},
\)
and
\(
 |\tau(z)-\overline{\tau(a)}|^2
 =\frac{4|1-\overline a z|^2}
        {|1-z|^2|1-a|^2}
\)
give
\[
 \int_{\mathbb H}
 \frac{\operatorname{Im}\tau(a)}
      {|w-\overline{\tau(a)}|^2}\,d\nu(w)
 =
 \frac14\int_\D
 \frac{1-|a|^2}{|1-\overline a z|^2}\,d\mu(z).
\]
Since $\tau$ maps $\D$ onto $\mathbb H$,
Lemma~\ref{lem:transport-kernel-tests} implies
\[
 \|\nu\|_{\CM(\mathbb H)}
 \lesssim
 \sup_{b\in\mathbb H}
 \int_{\mathbb H}
 \frac{\operatorname{Im}b}{|w-\overline b|^2}\,d\nu(w)
 =\frac14
 \sup_{a\in\D}
 \int_\D
 \frac{1-|a|^2}{|1-\overline a z|^2}\,d\mu(z)
 \lesssim \|\mu\|_{\CM}\leq1.
\]
Thus there is an absolute constant $A\geq1$ such that
$\|\nu\|_{\CM(\mathbb H)}\leq A$.

We now apply Garnett's discretization in the form recorded in
\cite[p.~285, proof of Corollary~1]{GN1998} to $\nu/A$.
It provides positive measures
\(
 \sigma_n=\frac1{N(n)}
           \sum_{k=1}^{4N(n)}\widetilde\nu_{n,k},
 \widetilde\nu_{n,k}
 =\sum_j\operatorname{Im}w_{n,k,j}\,
                  \delta_{w_{n,k,j}},
\)
such that $\sigma_n\to\nu/A$ vaguely. Moreover, the sequences
\(
 W_{n,k}=(w_{n,k,j})_j
\)
have uniformly bounded interpolation constants: there exists
$L_\mu\geq1$, independent of $n$ and $k$, such that
\(
 \mathfrak I_{\mathbb H}(W_{n,k})\leq L_\mu\) for all \(n,k\).

Let $s=\tau^{-1}$, and define the inverse weighted transfer by
\(
 \mathcal R\sigma=s_*\bigl(2|s'|\sigma\bigr).
\)
Since
\(
 s(w)=\frac{w-i}{w+i},
 2|s'(w)|=\frac4{|w+i|^2},
\)
we have
\(
 \mathcal R\nu=\mu,
 2|s'(w)|\operatorname{Im}w=1-|s(w)|^2.
\)
Define
\(
 \lambda_{n,k,j}=s(w_{n,k,j}),
 \Lambda_{n,k}=(\lambda_{n,k,j})_j,
 \mu_n=\mathcal R(A\sigma_n).
\)
The preceding weight identity gives
\[
 \mathcal R\widetilde\nu_{n,k}
 =\sum_j(1-|\lambda_{n,k,j}|^2)
                  \delta_{\lambda_{n,k,j}}
 =\mu_{\Lambda_{n,k}},
\]
and hence
\(
 \mu_n=\frac{A}{N(n)}
       \sum_{k=1}^{4N(n)}\mu_{\Lambda_{n,k}}.
\)
Thus we may take
\(
 N_n=4N(n),
 c_{n,k}=\frac{A}{N(n)},
 \sum_{k=1}^{N_n}c_{n,k}=4A.
\)

We next verify vague convergence on the disk.
For $f\in C_c(\D)$, the function
\(
 g(w)=2|s'(w)|f(s(w))
\)
belongs to $C_c(\mathbb H)$, because its support is contained
in the compact set $\tau(\operatorname{supp}f)$.
Therefore
\[
 \int_\D f\,d\mu_n
 =A\int_{\mathbb H}g\,d\sigma_n
 \longrightarrow
 \int_{\mathbb H}g\,d\nu
 =\int_\D f\,d\mu.
\]
This proves $\mu_n\to\mu$ vaguely.

Finally, composition with $\tau$ gives an isometric
correspondence between $H^\infty(\mathbb H)$ and $H^\infty(\D)$.
Consequently,
\(
 \mathfrak I_\D(\Lambda_{n,k})
 =\mathfrak I_{\mathbb H}(W_{n,k})
 \leq L_\mu.
\)
Lemma~\ref{lem:transport-interpolation-bounds} now yields
\(
 \operatorname{sep}(\Lambda_{n,k})\geq L_\mu^{-1},
 \|\mu_{\Lambda_{n,k}}\|_{\CM}\leq CL_\mu^2\) for all \(n,k\).
Taking
\(
 \eta_\mu=L_\mu^{-1},
 M_\mu=CL_\mu^2,
 A_0=4A
\)
completes the proof. The first two constants may depend on
$\mu$, whereas $A_0$ is absolute.
\end{proof}

For a continuous map $\Phi:\D\to\D$, define
\(
 q_\Phi(z)=\frac{1-|\Phi(z)|^2}{1-|z|^2},
 T_\Phi\mu=\Phi_*(q_\Phi\mu).
\)
Thus, for a Borel set $E\subset\D$,
\(
 (T_\Phi\mu)(E)
 =\int_{\Phi^{-1}(E)}q_\Phi(z)\,d\mu(z).
\)
The canonical measures satisfy
\begin{equation}
\label{eq:transport-canonical-identity}
 T_\Phi\mu_\Lambda
 =\sum_j(1-|\Phi(\lambda_j)|^2)\delta_{\Phi(\lambda_j)}
 =\mu_{\Phi(\Lambda)}.
\end{equation}
With atoms counted according to their indices, this identity holds even
when $\Phi$ is not injective. We can now combine the metric and measure estimates to
characterize continuous preservers in one direction.

\begin{samepage}
\begin{thm}
\label{thm:transport-one-sided}
For a continuous map $\Phi:\D\to\D$, the following conditions are equivalent:
\begin{enumerate}
\renewcommand{\labelenumi}{\textup{(\roman{enumi})}}
 \item $\Lambda\in\Cseq$ implies $\Phi(\Lambda)\in\Cseq$ for every indexed
       sequence $\Lambda\subset\D$.
 \item $\Phi$ is a homeomorphism of $\D$ onto itself, $\Phi^{-1}$ is
       uniformly continuous with respect to $\rho$, and there is a constant
       $C_\Phi<\infty$ such that
       \begin{equation}
       \label{eq:transport-positive-measure-bound}
       \|T_\Phi\mu\|_{\CM}\leq C_\Phi\|\mu\|_{\CM}
       \end{equation}
       for every positive Carleson measure $\mu$.
\end{enumerate}
\end{thm}
\end{samepage}

\begin{proof}
Assume~(i). Proposition~\ref{prop:transport-injective-proper}
shows that $\Phi$ is injective and satisfies
\(
 |z_n|\rightarrow1
 \Longrightarrow
 |\Phi(z_n)|\rightarrow1.
\)
Since $\Phi$ is continuous and injective, invariance of domain
implies that $\Phi$ is an open map.
For every compact set $K\subset\D$, the above boundary behavior
forces $\Phi^{-1}(K)$ to lie in a compact subdisk of $\D$.
Moreover, $\Phi^{-1}(K)$ is closed by continuity, and hence compact.
Thus $\Phi$ is proper and its image is closed in $\D$.
Since $\Phi(\D)$ is nonempty and both open and closed in the
connected disk $\D$, it follows that $\Phi(\D)=\D$.
The open mapping property then gives continuity of the inverse.

Suppose that $\Phi^{-1}$ is not uniformly continuous with respect to $\rho$.
There are $\varepsilon_0>0$ and point pairs $z_n,w_n\in\D$ such that
\(
 \rho(z_n,w_n)\geq\varepsilon_0,
  \rho(\Phi(z_n),\Phi(w_n))\rightarrow0.
\)
If a subsequence of $(z_n)$ remained in a compact subdisk, a further
subsequence would satisfy $z_n\to z\in\D$. Continuity of $\Phi$ would imply
$\Phi(w_n)\to\Phi(z)$, and continuity of $\Phi^{-1}$ would then give
$w_n\to z$, a contradiction. Thus $|z_n|\to1$; the same argument gives
$|w_n|\to1$. Passing to a subsequence, arrange that
\[
 \max\{\delta(z_{n+1}),\delta(w_{n+1})\}
 \leq\tfrac14\min\{\delta(z_n),\delta(w_n)\}.
\]
By Lemma~\ref{lem:transport-sparse-packets}, the sequence
$\Gamma=(z_1,w_1,z_2,w_2,\ldots)$ is interpolating, whereas its image is not
pairwise separated. This contradicts~(i).

We next show that $T_\Phi$ maps every positive Carleson measure to a Carleson
measure, initially without a bound uniform in that measure. Fix
$\|\mu\|_{\CM}\leq1$ and take the approximation in
Lemma~\ref{lem:transport-garnett-discretization}. Proposition~\ref{prop:transport-uniformization}
and \eqref{eq:transport-canonical-identity} yield
\(
 \|T_\Phi\mu_{\Lambda_{n,k}}\|_{\CM}
 \leq B(\Phi,\eta_\mu,M_\mu).
\)
Consequently,
\begin{equation}
\label{eq:transport-approximant-bound}
 \|T_\Phi\mu_n\|_{\CM}
 \leq B_{\Phi,\mu}
 :=A_0B(\Phi,\eta_\mu,M_\mu).
\end{equation}
Since $\Phi$ is proper, $q_\Phi(f\circ\Phi)\in C_c(\D)$ for every
$f\in C_c(\D)$. Hence
\[
 \int f\,d(T_\Phi\mu_n)
 =\int q_\Phi(f\circ\Phi)\,d\mu_n
 \longrightarrow
 \int q_\Phi(f\circ\Phi)\,d\mu
 =\int f\,d(T_\Phi\mu).
\]
All these measures are locally finite: a compact target set has compact
preimage, on which $q_\Phi$ is bounded. Thus $T_\Phi\mu_n\to T_\Phi\mu$
vaguely.

We now pass the Carleson bound to the limit. Write
\(
 \nu_n=T_\Phi\mu_n,
 \nu=T_\Phi\mu,
 B=B_{\Phi,\mu}.
\)
Then \(\nu_n\to\nu\) vaguely and
\(\|\nu_n\|_{\CM}\leq B\) for every \(n\).

For every open set \(G\subset\D\), we have
\(
 \nu(G)\leq\liminf_{n\to\infty}\nu_n(G).
\)
Indeed, if $f\in C_c(G)$ satisfies $0\leq f\leq1$,
we may extend $f$ by zero to a function in $C_c(\D)$.
Vague convergence gives
\[
 \int_G f\,d\nu
 =\lim_{n\to\infty}\int_G f\,d\nu_n
 \leq\liminf_{n\to\infty}\nu_n(G).
\]
Taking the supremum over such functions $f$ proves the
claim by inner regularity of the locally finite Borel
measure $\nu$.

Fix an arc $I\subset\T$ with $0<|I|<1$.
Choose slightly larger concentric arcs $I_\epsilon$
such that
\(
 |I|<|I_\epsilon|<1,
 |I_\epsilon|\downarrow|I|,
 S(I)\subset G_\epsilon
 :=\operatorname{int}_\D S(I_\epsilon).
\)
The inequality for open sets and the uniform Carleson bound give
\[
 \nu(S(I))
 \leq \nu(G_\epsilon)
 \leq \liminf_{n\to\infty}\nu_n(G_\epsilon)
 \leq \liminf_{n\to\infty}\nu_n(S(I_\epsilon))
 \leq B|I_\epsilon|.
\]
Letting $\epsilon\downarrow0$, we obtain
\(
 \nu(S(I))\leq B|I|.
\)

For $I=\T$, apply the same inequality for open sets to
$G=\D$. Since $S(\T)=\D$, we have
\(
 \nu(\D)
 \leq\liminf_{n\to\infty}\nu_n(\D)
 \leq B.
\)
Consequently,
\(
 \|T_\Phi\mu\|_{\CM}=\|\nu\|_{\CM}\leq B_{\Phi,\mu}.
\)
Thus $T_\Phi\mu$ is a Carleson measure.
By homogeneity, this conclusion holds for every positive
Carleson measure.

Positivity now upgrades this qualitative conclusion to the uniform bound
\eqref{eq:transport-positive-measure-bound}. If no such bound existed, we
could choose positive measures $\sigma_m$ with
\(
 \|\sigma_m\|_{\CM}\leq1,
 \|T_\Phi\sigma_m\|_{\CM}>4^m.
\)
The positive measure $\sigma=\sum_{m\geq1}2^{-m}\sigma_m$ satisfies
$\|\sigma\|_{\CM}\leq1$, so the preceding argument gives
$T_\Phi\sigma\in\CM$. On the other hand, monotone convergence yields
\[
 T_\Phi\sigma
 =\sum_{m\geq1}2^{-m}T_\Phi\sigma_m
 \geq2^{-m}T_\Phi\sigma_m
 \qquad(m\geq1).
\]
Thus $\|T_\Phi\sigma\|_{\CM}>2^m$ for every $m$, a contradiction.
Homogeneity proves \eqref{eq:transport-positive-measure-bound}.

Conversely, assume~(ii) and let $\Lambda\in\Cseq$. If
$\operatorname{sep}(\Lambda)\geq\eta>0$, uniform continuity of
$\Phi^{-1}$ provides $\varepsilon>0$ such that
\(
 \rho(\Phi(z),\Phi(w))<\varepsilon
 \Longrightarrow\rho(z,w)<\eta.
\)
Hence $\operatorname{sep}(\Phi(\Lambda))\geq\varepsilon$. Equations
\eqref{eq:transport-canonical-identity} and
\eqref{eq:transport-positive-measure-bound} imply
$\mu_{\Phi(\Lambda)}\in\CM$. The interpolation theorem now gives
$\Phi(\Lambda)\in\Cseq$.
\end{proof}

\begin{remark}
\label{rem:transport-distinct-roles}
The two quantitative requirements in
Theorem~\ref{thm:transport-one-sided} have distinct roles: uniform continuity
of the inverse preserves pairwise separation, while boundedness of the
weighted pushforward preserves the Carleson measure condition. The factor
$q_\Phi$ converts the canonical source weights into the corresponding image
weights; no uniform upper or positive lower bound on $q_\Phi$ is imposed.
\end{remark}

\begin{samepage}
\begin{thm}
\label{thm:transport-two-sided}
For a continuous map $\Phi:\D\to\D$, the following conditions are equivalent:
\begin{enumerate}
\renewcommand{\labelenumi}{\textup{(\roman{enumi})}}
 \item For every indexed sequence $\Lambda\subset\D$,
       \(
        \Lambda\in\Cseq\Longleftrightarrow
        \Phi(\Lambda)\in\Cseq.
       \)
 \item $\Phi$ is a homeomorphism of $\D$ onto itself, both $\Phi$ and
       $\Phi^{-1}$ are uniformly continuous with respect to $\rho$, and
       $T_\Phi,T_{\Phi^{-1}}$ act boundedly on the cone of positive Carleson
       measures.
\end{enumerate}
\end{thm}
\end{samepage}

\begin{proof}
Assume~(i). Theorem~\ref{thm:transport-one-sided} implies
that $\Phi$ is a homeomorphism of $\D$ onto itself,
$\Phi^{-1}$ is uniformly continuous with respect to $\rho$,
and $T_\Phi$ is bounded on the cone of positive Carleson measures.

For every interpolating sequence $\Gamma$, applying~(i) to
$\Lambda=\Phi^{-1}(\Gamma)$ shows that $\Lambda$ is interpolating.
Thus $\Phi^{-1}$ also preserves interpolating sequences.
Applying Theorem~\ref{thm:transport-one-sided} to $\Phi^{-1}$
shows that $\Phi$ is uniformly continuous with respect to $\rho$
and that $T_{\Phi^{-1}}$ is bounded on the same cone.
This proves~(ii).

Conversely, assume~(ii).
Applying Theorem~\ref{thm:transport-one-sided} separately to
$\Phi$ and $\Phi^{-1}$ shows that both maps preserve
interpolating sequences. Since $\Phi$ is bijective, this
gives both implications in~(i).
\end{proof}

\begin{samepage}
\begin{cor}
\label{cor:transport-continuous-group}
The class $\Cbi$ is a group under composition. For $\Phi,\Theta\in\Cbi$,
\[
 q_{\Phi\circ\Theta}(z)=q_\Phi(\Theta(z))q_\Theta(z),
 \qquad T_{\Phi\circ\Theta}=T_\Phi T_\Theta.
\]
\end{cor}
\end{samepage}

\begin{proof}
The identity map belongs to $\Cbi$.
If $\Phi,\Theta\in\Cbi$, then $\Phi\circ\Theta$ is continuous,
and for every indexed sequence $\Lambda\subset\D$,
\(
 \Lambda\in\Cseq
 \Longleftrightarrow
 \Theta(\Lambda)\in\Cseq
 \Longleftrightarrow
 (\Phi\circ\Theta)(\Lambda)\in\Cseq.
\)
Thus $\Phi\circ\Theta\in\Cbi$.

By Theorem~\ref{thm:transport-two-sided}, every $\Phi\in\Cbi$
is a homeomorphism of $\D$ onto itself.
For any indexed sequence $\Gamma\subset\D$, applying the
preservation equivalence for $\Phi$ to $\Phi^{-1}(\Gamma)$ gives
\(
 \Phi^{-1}(\Gamma)\in\Cseq
 \Longleftrightarrow
 \Gamma\in\Cseq.
\)
Since $\Phi^{-1}$ is continuous, it follows that
$\Phi^{-1}\in\Cbi$.
Together with associativity of composition, this proves
the group assertion.

For the weights, their definition gives
\[
 q_{\Phi\circ\Theta}(z)
 =\frac{1-|\Phi(\Theta(z))|^2}{1-|z|^2}
 =\frac{1-|\Phi(\Theta(z))|^2}{1-|\Theta(z)|^2}
  \frac{1-|\Theta(z)|^2}{1-|z|^2}
 =q_\Phi(\Theta(z))q_\Theta(z).
\]
Now let $\mu$ be a positive Carleson measure and let $f$ be
a bounded nonnegative Borel function on $\D$.
By the definition of the weighted pushforward,
\begin{align*}
 \int_\D f\,d(T_\Phi T_\Theta\mu)
 &=\int_\D f(\Phi(w))q_\Phi(w)\,d(T_\Theta\mu)(w)
 =\int_\D f(\Phi(\Theta(z)))
       q_\Phi(\Theta(z))q_\Theta(z)\,d\mu(z)\\
 &=\int_\D f((\Phi\circ\Theta)(z))
       q_{\Phi\circ\Theta}(z)\,d\mu(z)
 =\int_\D f\,d(T_{\Phi\circ\Theta}\mu).
\end{align*}
Since this holds for every such $f$ and $\mu$, we obtain
$T_{\Phi\circ\Theta}=T_\Phi T_\Theta$
on the cone of positive Carleson measures.
\end{proof}

\begin{example}
\label{ex:one-sided-not-two-sided}
The map
\(
 F(z)=\frac{2z}{1+|z|^2}, z\in\overline\D,
\)
is a homeomorphism of the closed disk, extends smoothly to a neighborhood
of $\overline\D$, and has identity boundary values. Nevertheless,
\(
 F|_{\D}\in\Cplus\setminus\Cbi.
\)
In particular, $\Cbi\subsetneq\Cplus$, even among maps that are smooth
on the closed disk.

Indeed, the radial function $f(r)=2r/(1+r^2)$ is strictly increasing
from $0$ to $1$. Write $d_\D=2\operatorname{arctanh}\rho$ for the
hyperbolic distance. In hyperbolic polar coordinates
$R=2\operatorname{arctanh}r$, the map $F$ takes $(R,\theta)$ to
$(2R,\theta)$. With respect to the hyperbolic line element
\(d\ell^2=dR^2+\sinh^2R\,d\theta^2\), the radial and tangential
stretch factors of \(F^{-1}(R,\theta)=(R/2,\theta)\) are,
respectively,
\[
 \frac12,
 \qquad
 \frac{\sinh(R/2)}{\sinh R}
 =\frac{1}{2\cosh(R/2)}
 \leq\frac12.
\]
The same bound holds at the origin, where $D(F^{-1})(0)=I/2$.
Integrating along curves shows that $F^{-1}$ is $1/2$-Lipschitz for
$d_\D$. Consequently,
\(
 d_\D(F(z),F(w))\geq2d_\D(z,w),
 \rho(F(z),F(w))\geq\rho(z,w).
\)
Every finite partial product in the product separation condition can
therefore only increase under $F$. Passing to the infinite products
shows that $F$ preserves every interpolating sequence.

To see that preservation does not hold in both directions, let $r\uparrow1$ and set
\(
 t=1-r, z_r=r, w_r=re^{it^2},
 s_r=\frac{2r}{1+r^2}.
\)
Then $|z_r-w_r|\sim t^2$ and $1-r^2\sim2t$, whereas
\(
 1-s_r^2=\frac{(1-r^2)^2}{(1+r^2)^2}\sim t^2,
  |F(z_r)-F(w_r)|\sim t^2.
\)
Using
$|1-\overline zw|^2=(1-|z|^2)(1-|w|^2)+|z-w|^2$, we obtain
\(
 \rho(z_r,w_r)\rightarrow0,
 \rho(F(z_r),F(w_r))\rightarrow\frac1{\sqrt2}.
\)
Thus $F$ is not uniformly continuous with respect to $\rho$, and
Theorem~\ref{thm:transport-two-sided} gives $F|_{\D}\notin\Cbi$.
Finally,
\[
 q_F(z)=\frac{1-|z|^2}{(1+|z|^2)^2}\longrightarrow0
 \qquad(|z|\to1),
\]
so preservation in one direction does not force a positive lower bound for the
boundary defect ratio, even with smoothness and identity boundary values.
\end{example}

\begin{example}
\label{ex:one-sided-no-radial-limits}
There exists a $C^\infty$ diffeomorphism $F:\D\to\D$ such
that $F\in\Cplus\setminus\Cbi$ and $F$ has no radial limit
at any point of $\T$. In fact, for every $\zeta\in\T$,
every point of $\T$ occurs as a limit of $F(r_n\zeta)$ for
some sequence $r_n\uparrow1$.

Define $F(0)=0$ and
\begin{equation}
\label{eq:one-sided-no-radial-limits-map}
 F(re^{i\theta})
 =\tanh\!\left(\frac{2r}{1-r^2}\right)
  \exp\!\left(i\theta+i\log\frac{1+r^2}{1-r^2}\right),
 \qquad 0<r<1.
\end{equation}
In hyperbolic polar coordinates $R=2\operatorname{arctanh}r$,
this becomes
\[
 F(R,\theta)
 =\bigl(2\sinh R,\,\theta+\log\cosh R\bigr).
\]
The radial coordinate increases strictly from $0$ to
$\infty$, so $F$ is a homeomorphism of $\D$ onto itself.
It is smooth with invertible differential away from the
origin. Moreover, $\tanh(2r/(1-r^2))/r$ extends to an even
real analytic function near $0$ with value $2$, while
$\log((1+r^2)/(1-r^2))$ is even and real analytic there.
Consequently, near $0$ the map has the form
$F(z)=A(|z|^2)z$, where $A$ is smooth and $A(0)=2$.
Thus $DF(0)=2I$, and $F$ is a $C^\infty$ diffeomorphism
of $\D$.

We first prove preservation in one direction.
For the hyperbolic line element
$d\ell^2=dR^2+\sinh^2R\,d\theta^2$, the matrix of
$D(F^{-1})$ at $F(R,\theta)$, in the corresponding
hyperbolic orthonormal frames, is
\[
 \begin{pmatrix}
  \dfrac{1}{2\cosh R}&0\\[6pt]
  -\dfrac12\tanh^2R&
  \dfrac{\sinh R}{\sinh(2\sinh R)}
 \end{pmatrix}.
\]
Since $\sinh(2x)\geq2x$ for $x\geq0$, the lower right
entry is at most $1/2$ for $R>0$.
Writing $u=\tanh^2R\in[0,1)$, the sum of the squares
of the entries is at most
\[
 \frac14(1-u+u^2)+\frac14\leq\frac12.
\]
The operator norm of $D(F^{-1})$ in the hyperbolic metric
is therefore at most $1/\sqrt2$.
The same bound holds at the origin, where
$D(F^{-1})(0)=I/2$.
Integrating along curves shows that $F^{-1}$ is
$1/\sqrt2$-Lipschitz for $d_\D$. Hence
\[
 d_\D(F(z),F(w))\geq\sqrt2\,d_\D(z,w),
 \qquad
 \rho(F(z),F(w))\geq\rho(z,w),
 \qquad z,w\in\D.
\]
Every finite partial product in Carleson's product separation
condition can therefore only increase under $F$.
Passing to the infinite products proves that $F\in\Cplus$.

The map $F$ is not uniformly continuous with respect to
$\rho$. Indeed, for $R>0$, set
$z_R=\tanh(R/2)$ and
$w_R=\tanh((R+e^{-R})/2)$.
Then $d_\D(z_R,w_R)=e^{-R}\to0$, whereas the triangle
inequality gives
\[
 d_\D(F(z_R),F(w_R))
 \geq2\bigl(\sinh(R+e^{-R})-\sinh R\bigr).
\]
The right side tends to $1$ as $R\to\infty$.
Thus $\rho(z_R,w_R)\to0$ while the image distances stay
bounded away from zero.
Theorem~\ref{thm:transport-two-sided} implies that
$F\notin\Cbi$.

Finally, fix $\zeta\in\T$ and $\alpha\in[0,2\pi)$.
Choose $R_n>0$ so that
$\log\cosh R_n=2\pi n+\alpha$, and put
$r_n=\tanh(R_n/2)$.
Then $r_n\uparrow1$ and
\[
 F(r_n\zeta)
 =\tanh(\sinh R_n)e^{i\alpha}\zeta
 \longrightarrow e^{i\alpha}\zeta.
\]
Since $|F(r\zeta)|\to1$ as $r\uparrow1$, the radial
cluster set is exactly $\T$.
In particular, $F$ has no radial limit at any point of
$\T$ and admits no continuous extension to $\overline\D$.
\end{example}

\subsection{Boundary characterization of bidirectional preservers}
\label{subsec:boundary-characterization}
We next express preservation in both directions in terms of boundary geometry.
Recall that $d_\D=2\operatorname{arctanh}\rho$ is the hyperbolic distance.
The identity $\rho=\tanh(d_\D/2)$ shows that these two distances induce
the same uniform structure.
A homeomorphism $F:\D\to\D$ will be called a
\emph{uniform homeomorphism with respect to $\rho$}, or a
\emph{$\rho$-uniform homeomorphism}, if both $F$ and $F^{-1}$
are uniformly continuous with respect to $\rho$.

A circle homeomorphism $h:\T\to\T$ is bi-Lipschitz if there
exists $L\geq1$ such that
\[
 L^{-1}d_\T(\zeta,\eta)
 \leq d_\T(h(\zeta),h(\eta))
 \leq Ld_\T(\zeta,\eta)
 \qquad(\zeta,\eta\in\T).
\]
We denote this group by $\operatorname{BiLip}(\T)$ and its subgroup of maps
that preserve orientation by $\operatorname{BiLip}^{+}(\T)$.
A circle homeomorphism $h:\T\to\T$ is quasisymmetric if there
exists $M\geq1$ such that
\(
 M^{-1}\leq\frac{|h(I)|}{|h(J)|}\leq M
\)
for every pair of nondegenerate adjacent arcs $I,J\subset\T$
with disjoint interiors and $|I|=|J|$.
Both orientations are allowed.

A circle homeomorphism $h$ that preserves orientation is
\emph{strongly quasisymmetric} if there exist $A\geq1$
and $\alpha>0$ such that, for every nondegenerate arc
$I\subset\T$ and every Borel set $E\subset I$,
\begin{equation}
\label{eq:sqs-measure-definition}
 \frac{|h(E)|}{|h(I)|}
 \leq A\left(\frac{|E|}{|I|}\right)^\alpha.
\end{equation}
Equivalently, an angular lift $H:\mathbb R\to\mathbb R$,
satisfying
\(
 h(e^{it})=e^{iH(t)},
  H(t+2\pi)=H(t)+2\pi,
\)
is locally absolutely continuous and $H'$ is an $A_\infty$ weight.
These maps form a group under composition
\cite[pp.~921--924]{FanHuShen2017}.
By absolute continuity, \eqref{eq:sqs-measure-definition}
also holds for every Lebesgue measurable subset $E\subset I$.
We also include maps that reverse orientation, obtained by composition with
complex conjugation, and denote the resulting class by
$\operatorname{SQS}^{\pm}(\T)$.
Every map in $\operatorname{SQS}^{\pm}(\T)$ is quasisymmetric, and
\(
 \operatorname{BiLip}(\T)
 \subset\operatorname{SQS}^{\pm}(\T).
\)

We use the classical characterization of quasiconformal preservers
due to Astala--Zinsmeister and Gonz\'alez--Nicolau:
a quasiconformal homeomorphism of $\D$ that preserves orientation maps every
interpolating sequence to an interpolating sequence
if and only if its boundary map is strongly quasisymmetric
\cite{AZ1991,GN1998}.
The formulation on the upper half plane appears in
\cite[p.~284, Theorem]{GN1998}; the disk formulation follows
by passage to the upper half plane using Cayley transforms,
after normalization at a boundary point.
Since strongly quasisymmetric homeomorphisms form a group,
the same criterion applies to the inverse map.
Consequently, preservation in one direction is equivalent to preservation in
both directions within the quasiconformal class.
The corresponding statement for maps that reverse orientation
follows by composition with complex conjugation.

\begin{samepage}
\begin{lem}
\label{lem:biuniform-from-preservation}
Every map in $\Cbi$ is a $\rho$-uniform homeomorphism of $\D$.
\end{lem}
\end{samepage}

This follows immediately from Theorem~\ref{thm:transport-two-sided}.

To obtain a boundary characterization, we first relate
$\rho$-uniformity to hyperbolic quasi-isometries and
boundary extensions.

\begin{samepage}
\begin{lem}
\label{lem:uniform-boundary-rigidity}
Let $F:\D\to\D$ be a $\rho$-uniform homeomorphism.
Then $F$ is a quasi-isometry for $d_\D$ and extends
uniquely to a homeomorphism of $\overline\D$.
Its boundary map $h_F$ is quasisymmetric.
If $h_F=\operatorname{id}_\T$, then
\(
  \sup_{z\in\D}d_\D(F(z),z)<\infty.
\)
\end{lem}
\end{samepage}

\begin{proof}
Since $\rho$ and $d_\D$ induce the same uniform structure,
both $F$ and $F^{-1}$ are uniformly continuous with respect
to $d_\D$.
Thus there exists $a>0$ such that
\[
  d_\D(z,w)\leq a
  \quad\Longrightarrow\quad
  d_\D(F(z),F(w))\leq 1.
\]
Subdividing a hyperbolic geodesic from $z$ to $w$ into
segments of length at most $a$ and applying the triangle
inequality gives
\(
  d_\D(F(z),F(w))
  \leq a^{-1}d_\D(z,w)+1.
\)
Applying the same argument to $F^{-1}$, we obtain
constants $A\geq1$ and $B\geq0$ such that
\[
  A^{-1}d_\D(z,w)-B
  \leq d_\D(F(z),F(w))
  \leq A d_\D(z,w)+B
  \qquad (z,w\in\D).
\]
Together with surjectivity, these inequalities show that
$F$ is a quasi-isometry of the hyperbolic disk.

By the boundary extension theorem for quasi-isometries
\cite[Theorem~5.35]{Vaisala2005}, applied with source
basepoint $0$ and target basepoint $F(0)$, the map $F$
induces a homeomorphism $h_F:\T\to\T$ that is
quasisymmetric with respect to visual metrics.
For a common sufficiently small visual parameter
$\varepsilon>0$, the visual metrics based at $0$ and
$F(0)$ are each comparable to $|\zeta-\eta|^\varepsilon$,
with comparison constants allowed to depend on the
basepoint.
Hence $h_F$ is quasisymmetric for the chordal metric.
Since the chordal metric is bi-Lipschitz equivalent to
$d_\T$, the map $h_F$ is also quasisymmetric for $d_\T$.

Under the standard identification of the Gromov boundary
of the hyperbolic disk with $\T$, convergence to a boundary
point agrees with Euclidean convergence.
The boundary extension theorem, together with continuity
of $F$ in $\D$, therefore gives a continuous extension
to $\overline\D$.
Applying the same argument to $F^{-1}$ gives a continuous
extension of its inverse.
The two extensions are mutually inverse, since their
compositions agree with the identity on the dense subset
$\D$ of $\overline\D$.
Thus the extension of $F$ is a homeomorphism.
Uniqueness also follows from the density of $\D$
in $\overline\D$.

Finally, assume that $h_F=\operatorname{id}_\T$.
The hyperbolic disk is a proper geodesic Gromov hyperbolic
space with a pole at $0$, since every point lies on a
geodesic ray issuing from $0$.
Its visual boundary is uniformly perfect.
The boundary rigidity theorem \cite{LiangZhou2022}
therefore implies that $F$ lies at a bounded distance
from the identity.
This proves \(\sup_{z\in\D}d_\D(F(z),z)<\infty\).
\end{proof}

We recall the following characterization of the boundary
kernel, expressed in terms of interpolating sequences and
radial boundary values; see
\cite[Theorems~2.8 and~5.6, Corollary~3.7]{Wu2026}.

\begin{lem}
\label{lem:wu-kernel-criterion}
Let $\Psi:\D\to\D$ be an arbitrary map.
The following three properties hold simultaneously if and only if
conditions \textup{(K0)}--\textup{(K3)} below are satisfied:
\(
 \Lambda\in\Cseq
 \Longleftrightarrow
 \Psi(\Lambda)\in\Cseq\) for every indexed sequence \(\Lambda\) in \(\D\);
\(
 c(1-|z|^2)\leq1-|\Psi(z)|^2\leq C(1-|z|^2)\) for \(z\in\D\),
for some constants $0<c\leq C<\infty$; and
\(
 \lim_{r\uparrow1}\Psi(r\zeta)=\zeta\) for \(\zeta\in\T\).
\begin{enumerate}
\item[\textup{(K0)}]
$\Psi$ is injective.

\item[\textup{(K1)}]
For every sequence $(z_n)$ in $\D$,
\(
 |z_n|\longrightarrow1
 \Longleftrightarrow
 |\Psi(z_n)|\longrightarrow1.
\)

\item[\textup{(K2)}]
There exist constants $0<M<r_0<1$ such that
\(
 \rho(\Psi(z),z)\leq M\) whenever \(r_0\leq|z|<1\).

\item[\textup{(K3)}]
For every pair of sequences $(z_n)$ and $(w_n)$ in $\D$
satisfying $|z_n|\to1$ and $|w_n|\to1$,
\[
 \rho(z_n,w_n)\longrightarrow0
 \quad\Longleftrightarrow\quad
 \rho(\Psi(z_n),\Psi(w_n))\longrightarrow0.
\]
\end{enumerate}
\end{lem}

\begin{samepage}
\begin{lem}
\label{lem:bounded-displacement-preserver}
Let $F:\D\to\D$ be a $\rho$-uniform homeomorphism such that
$\sup_{z\in\D}\rho(F(z),z)=M<1$. Then $F\in\Cbi$ and
\(
 1-|F(z)|\asymp1-|z|.
\)
\end{lem}
\end{samepage}

\begin{proof}
We verify conditions \textup{(K0)}--\textup{(K3)}
of Lemma~\ref{lem:wu-kernel-criterion}.

Since $F$ is a homeomorphism of $\D$, it is injective.
Moreover, both $F$ and $F^{-1}$ map compact subsets of $\D$
to compact subsets of $\D$. Consequently,
\(
 |z_n|\rightarrow1
 \Longleftrightarrow
 |F(z_n)|\rightarrow1,
\)
so \textup{(K0)} and \textup{(K1)} hold.

Choose constants $M<M_1<r_0<1$. Then
\(
 \rho(F(z),z)\leq M<M_1\) whenever \(|z|\geq r_0\),
which gives \textup{(K2)}.
Uniform continuity of $F$ and $F^{-1}$ with respect to $\rho$
gives
\(
 \rho(z_n,w_n)\rightarrow0
 \Longleftrightarrow
 \rho(F(z_n),F(w_n))\rightarrow0.
\)
In particular, \textup{(K3)} holds.

Lemma~\ref{lem:wu-kernel-criterion} now shows that $F$
preserves interpolating sequences in both directions and that
\(
 1-|F(z)|^2\asymp1-|z|^2.
\)
Since $F$ is continuous, $F\in\Cbi$.
Finally, the inequalities $1\leq1+|z|<2$ show that this
estimate is equivalent to \(
 1-|F(z)|\asymp1-|z|.
\)
\end{proof}

\begin{samepage}
\begin{lem}
\label{lem:qc-rho-uniform}
Let $G:\D\to\D$ be a $K$-quasiconformal homeomorphism,
where $K\geq1$. Then
\[
  \rho(G(z),G(w))
  \leq16\,\rho(z,w)^{1/K}
  \qquad (z,w\in\D).
\]
The same estimate holds for $G^{-1}$.
In particular, $G$ is a $\rho$-uniform homeomorphism.
\end{lem}
\end{samepage}

\begin{proof}
For $a\in\D$, let
\(
  \varphi_a(\zeta)=\frac{\zeta-a}{1-\overline a\zeta}.
\)
Fix $z\in\D$ and define
\(
  H_z=\varphi_{G(z)}\circ G\circ\varphi_z^{-1}.
\)
Then $H_z$ is a $K$-quasiconformal homeomorphism
of $\D$ satisfying $H_z(0)=0$.
Mori's theorem \cite[p.~157]{Mori1956} gives
\[
  |H_z(u)|\leq16\,|u|^{1/K}
  \qquad (u\in\D).
\]
Taking $u=\varphi_z(w)$, we obtain
\[
  \rho(G(z),G(w))
  =\bigl|H_z(\varphi_z(w))\bigr|
  \leq16\,|\varphi_z(w)|^{1/K}
   =16\,\rho(z,w)^{1/K}.
\]
Since $G^{-1}$ is also $K$-quasiconformal, the same
argument applies to $G^{-1}$.
These estimates imply uniform continuity of both maps
with respect to $\rho$.
\end{proof}

Comparison with a quasiconformal map having the same
boundary values now reduces bidirectional preservation
to strong quasisymmetry.

\begin{samepage}
\begin{thm}
\label{thm:continuous-bidirectional-boundary}
For a continuous map $\Phi:\D\to\D$, the following
conditions are equivalent:
\begin{enumerate}
\renewcommand{\labelenumi}{\textup{(\roman{enumi})}}
  \item $\Phi\in\Cbi$.
  \item $\Phi$ is a $\rho$-uniform homeomorphism of $\D$,
        and its boundary extension satisfies
        $h_\Phi\in\operatorname{SQS}^{\pm}(\T)$.
\end{enumerate}
\end{thm}
\end{samepage}

\begin{proof}
By Lemma~\ref{lem:biuniform-from-preservation},
condition~(i) implies that $\Phi$ is a $\rho$-uniform
homeomorphism of $\D$.
Since this property is also part of condition~(ii),
it suffices to prove
\(
  \Phi\in\Cbi
  \Longleftrightarrow
  h_\Phi\in\operatorname{SQS}^{\pm}(\T)
\)
for every $\rho$-uniform homeomorphism $\Phi$ of $\D$.

Let $J(z)=\overline z$.
Since
\(
  \rho(J(z),J(w))=\rho(z,w),
  1-|J(z)|^2=1-|z|^2,
\)
complex conjugation preserves pairwise separation
and the canonical weights.
Moreover, for every arc $I\subset\T$,
\(
  J(S(I))=S(J(I)),
   |J(I)|=|I|,
\)
so it also preserves the Carleson measure condition.
Carleson's interpolation theorem therefore gives
\(
  \Lambda\in\Cseq
  \Longleftrightarrow
  J(\Lambda)\in\Cseq.
\)
Since $J$ is continuous, we conclude that $J\in\Cbi$.
Composition with $J$ preserves $\rho$-uniformity and,
by Corollary~\ref{cor:transport-continuous-group},
membership in $\Cbi$.
Moreover,
\(
  h_{J\circ\Phi}=J|_\T\circ h_\Phi,
\)
and composition with $J|_\T$ preserves membership in
$\operatorname{SQS}^{\pm}(\T)$.
Replacing $\Phi$ by $J\circ\Phi$ when necessary,
we may therefore assume that $\Phi$ preserves orientation.

By Lemma~\ref{lem:uniform-boundary-rigidity},
$\Phi$ extends to a homeomorphism of $\overline\D$,
and its boundary map $h=h_\Phi$ is quasisymmetric.
The Beurling--Ahlfors extension theorem
\cite{AB1956} provides a quasiconformal
homeomorphism $G$ of $\D$ with boundary map $h$.
By Lemma~\ref{lem:qc-rho-uniform}, $G$ is a
$\rho$-uniform homeomorphism.

Set
\(
  \Theta=\Phi\circ G^{-1}.
\)
Both $\Theta$ and
$\Theta^{-1}=G\circ\Phi^{-1}$ are uniformly continuous
with respect to $\rho$.
Furthermore, their continuous boundary extensions give
\(
  h_\Theta=h_\Phi\circ h_G^{-1}
           =h\circ h^{-1}
           =\operatorname{id}_\T.
\)
Lemma~\ref{lem:uniform-boundary-rigidity} therefore yields
\(
  D_\Theta:=
  \sup_{z\in\D}d_\D(\Theta(z),z)<\infty.
\)
Using $\rho=\tanh(d_\D/2)$, we obtain
\[
  \sup_{z\in\D}\rho(\Theta(z),z)
  \leq\tanh(D_\Theta/2)<1.
\]
Thus Lemma~\ref{lem:bounded-displacement-preserver}
applies and gives $\Theta\in\Cbi$.

Since $\Cbi$ is a group by
Corollary~\ref{cor:transport-continuous-group},
the identities
\(
  \Phi=\Theta\circ G,
  G=\Theta^{-1}\circ\Phi
\)
imply
\(
  \Phi\in\Cbi
  \Longleftrightarrow
  G\in\Cbi.
\)

It remains to characterize the latter condition.
If $G\in\Cbi$, then $G$ maps every interpolating
sequence to an interpolating sequence.
The quasiconformal characterization recalled above
\cite{AZ1991,GN1998} implies that $h$ is strongly
quasisymmetric.

Conversely, suppose that $h$ is strongly quasisymmetric.
Since strongly quasisymmetric homeomorphisms form a group,
$h^{-1}$ is also strongly quasisymmetric.
Applying the same characterization separately to $G$
and $G^{-1}$ shows that both maps preserve interpolating
sequences.
Hence $G\in\Cbi$.

We have therefore proved, for maps that preserve orientation,
that $\Phi\in\Cbi$ if and only if $h_\Phi$ is strongly
quasisymmetric.
The reduction by complex conjugation gives the stated
equivalence for both orientations.
\end{proof}

The preceding comparison also identifies the additional
metric condition that turns preservation in one direction
into preservation in both directions.

\begin{cor}
\label{cor:automatic-bidirectional-preservation}
Let $\Phi\in\Cplus$. Then
 \(\Phi\in\Cbi\) if and only if
 \(\Phi\) is uniformly continuous with respect to \(\rho\).
\end{cor}

\begin{proof}
Necessity follows from Theorem~\ref{thm:transport-two-sided}.
For sufficiency, assume that $\Phi\in\Cplus$ is uniformly
continuous with respect to $\rho$.
Theorem~\ref{thm:transport-one-sided} shows that $\Phi$ is a
homeomorphism of $\D$ onto itself and that $\Phi^{-1}$ is
uniformly continuous with respect to $\rho$.
Thus $\Phi$ is a $\rho$-uniform homeomorphism.
Since complex conjugation belongs to $\Cbi$ and is an
isometry for $\rho$, we may assume that $\Phi$ preserves
orientation.

By Lemma~\ref{lem:uniform-boundary-rigidity}, the boundary
map $h=h_\Phi$ is quasisymmetric. Choose a quasiconformal
homeomorphism $G:\D\to\D$ with boundary map $h$, using
\cite{AB1956}, and set $\Theta=\Phi\circ G^{-1}$.
Lemma~\ref{lem:qc-rho-uniform} shows that $\Theta$ is a
$\rho$-uniform homeomorphism with identity boundary values.
Lemmas~\ref{lem:uniform-boundary-rigidity}
and~\ref{lem:bounded-displacement-preserver} therefore give
$\Theta\in\Cbi$.
Since $\Theta^{-1}\in\Cbi$ and $\Phi\in\Cplus$, the identity
$G=\Theta^{-1}\circ\Phi$ shows that $G$ maps every
interpolating sequence to an interpolating sequence.
The quasiconformal characterization \cite{AZ1991,GN1998}
implies that $h$ is strongly quasisymmetric.
Theorem~\ref{thm:continuous-bidirectional-boundary} now yields
$\Phi\in\Cbi$.
\end{proof}

The preceding results also give a boundary characterization of
weighted Carleson measure preservation within the class of
$\rho$-uniform homeomorphisms.

\begin{cor}
\label{cor:weighted-carleson-boundary}
Let $\Phi:\D\to\D$ be a $\rho$-uniform homeomorphism, and let
$h_\Phi:\T\to\T$ be its boundary map. The following conditions
are equivalent:
\begin{enumerate}
\renewcommand{\labelenumi}{\textup{(\roman{enumi})}}
\item $T_\Phi\mu\in\CM$ for every positive Carleson measure $\mu$.
\item There exists a constant $C<\infty$ such that
\(
  \|T_\Phi\mu\|_{\CM}\le C\|\mu\|_{\CM}
\)
for every positive Carleson measure $\mu$.
\item Both $T_\Phi$ and $T_{\Phi^{-1}}$ act boundedly on the
cone of positive Carleson measures.
\item $h_\Phi\in\operatorname{SQS}^{\pm}(\T)$.
\end{enumerate}
\end{cor}

\begin{proof}
The implications
\textup{(iii)}$\Rightarrow$\textup{(ii)}$\Rightarrow$\textup{(i)}
are immediate.

To prove \textup{(i)}$\Rightarrow$\textup{(ii)}, suppose that
\textup{(i)} holds but no uniform bound exists. Choose positive
Carleson measures $\mu_n$ such that
\(
  \|\mu_n\|_{\CM}\le1,
  \|T_\Phi\mu_n\|_{\CM}>4^n.
\)
Then $\mu=\sum_{n\ge1}2^{-n}\mu_n$ is a positive Carleson measure.
Positivity and monotone convergence give
\(
  T_\Phi\mu\ge2^{-n}T_\Phi\mu_n,
\)
so $\|T_\Phi\mu\|_{\CM}>2^n$ for every $n$, contradicting
\textup{(i)}.

Assume \textup{(ii)}. Since $\Phi$ is a $\rho$-uniform
homeomorphism, Theorem~\ref{thm:transport-one-sided} gives
$\Phi\in\Cplus$. The uniform continuity of $\Phi$ with respect
to $\rho$, together with
Corollary~\ref{cor:automatic-bidirectional-preservation},
then yields $\Phi\in\Cbi$.
Theorem~\ref{thm:continuous-bidirectional-boundary} proves
\textup{(iv)}.

Conversely, \textup{(iv)} and
Theorem~\ref{thm:continuous-bidirectional-boundary} imply
$\Phi\in\Cbi$. Theorem~\ref{thm:transport-two-sided} therefore
gives \textup{(iii)}.
\end{proof}

Our characterization of continuous preservers contains the classical
quasiconformal criterion as a special case. Indeed, quasiconformal
disk homeomorphisms are automatically $\rho$-uniform by
Lemma~\ref{lem:qc-rho-uniform}, so the metric condition in
Theorem~\ref{thm:continuous-bidirectional-boundary} is automatic
in this setting. The following corollary therefore recovers the
classical characterization, whose sufficiency was established by
Astala--Zinsmeister \cite{AZ1991} and whose necessity was proved by
Gonz\'alez--Nicolau \cite{GN1998}, and includes both orientations.

\begin{samepage}
\begin{cor}
\label{cor:classical-qc-specialization}
Let $\Phi:\D\to\D$ be a quasiconformal or antiquasiconformal
homeomorphism, and let $h_\Phi:\T\to\T$ be its boundary map.
The following conditions are equivalent:
\begin{enumerate}
\renewcommand{\labelenumi}{\textup{(\roman{enumi})}}
\item $\Phi\in\Cplus$.
\item $\Phi\in\Cbi$.
\item $h_\Phi\in\operatorname{SQS}^{\pm}(\T)$.
\end{enumerate}
\end{cor}
\end{samepage}

\begin{proof}
Let $J(z)=\overline z$.
Complex conjugation preserves Carleson sequences in both directions,
and composition with $J|_{\T}$ preserves membership in
$\operatorname{SQS}^{\pm}(\T)$.
Thus each of \textup{(i)}--\textup{(iii)} is unchanged when
$\Phi$ is replaced by $J\circ\Phi$.
We may therefore assume that $\Phi$ preserves orientation.

By Lemma~\ref{lem:qc-rho-uniform}, both $\Phi$ and $\Phi^{-1}$
are uniformly continuous with respect to $\rho$.
Theorem~\ref{thm:continuous-bidirectional-boundary} consequently
gives the equivalence of \textup{(ii)} and \textup{(iii)}.
The implication \textup{(ii)}$\Rightarrow$\textup{(i)} is immediate.
Finally, if \textup{(i)} holds, the uniform continuity of $\Phi$
and Corollary~\ref{cor:automatic-bidirectional-preservation}
yield \textup{(ii)}.
\end{proof}

The next example shows that the continuous classification also
includes orientation-preserving maps that are not quasiconformal,
even when they are $C^1$ on the closed disk.

\begin{example}[A non-quasiconformal bidirectional preserver]
\label{ex:non-qc-bidirectional-preserver}
Define $f:[0,1]\to[0,1]$ by
\[
 f(r)=
 \begin{cases}
  0, & r=0,\\[2pt]
  \displaystyle\frac12\exp\!\left(1-\frac{1}{2r}\right),
     & 0<r\leq\frac12,\\[4pt]
  r, & \frac12<r\leq1,
 \end{cases}
\]
and set $F(0)=0$ and $F(re^{it})=f(r)e^{it}$ for $r>0$.
Then $F$ is an orientation-preserving $C^1$ homeomorphism of
$\overline\D$, its restriction to $\D$ belongs to $\Cbi$,
and $F|_{\D}$ is not quasiconformal.

Indeed, $f$ is a strictly increasing continuous bijection of
$[0,1]$ onto itself, and $F(z)=z$ for $|z|\geq1/2$.
For $0<r<1/2$,
\(
 f'(r)=\frac{f(r)}{2r^2},
 \lim_{r\downarrow0}\frac{f(r)}r=0,
 \lim_{r\downarrow0}f'(r)=0.
\)
Consequently, $DF(0)=0$ and $DF$ is continuous at the origin.
At $r=1/2$, the inner and outer values of $f$ and $f'$ agree:
they are $1/2$ and $1$, respectively.
Thus $F\in C^1(\overline\D)$.

Every Carleson sequence has only finitely many terms in
$\{z:|z|\leq1/2\}$.
Hence $F$ changes only finitely many terms of any such sequence,
and its injectivity prevents repetitions.
Lemma~\ref{lem:transport-finite-perturbation} therefore shows that
$F$ preserves Carleson sequences.
The inverse $F^{-1}$ is also the identity for $|z|\geq1/2$,
so the same argument applies to $F^{-1}$.
It follows that $F|_{\D}\in\Cbi$.

Finally, at $z=re^{it}$ with $0<r<1/2$, the radial and tangential
principal stretches of $F$ are $f'(r)$ and $f(r)/r$, respectively.
Their ratio is
\(
 \frac{f'(r)}{f(r)/r}=\frac{1}{2r}.
\)
This ratio has infinite essential supremum on $\D$,
whereas quasiconformality requires it to be essentially bounded.
Thus $F|_{\D}$ is not quasiconformal.
\end{example}

\begin{remark}
\label{rem:sqs-not-split}
The boundary correspondence
\(
  \partial:\Cbi\rightarrow\operatorname{SQS}^{\pm}(\T),
   \partial\Phi=h_\Phi,
\)
is a group homomorphism, since boundary extensions
respect composition.
It is surjective by
Theorem~\ref{thm:continuous-bidirectional-boundary}
and the Beurling--Ahlfors extension theorem \cite{AB1956},
where maps that reverse orientation are handled by composition with complex conjugation.
Thus there is an exact sequence
\[
  1\longrightarrow\ker\partial
  \longrightarrow\Cbi
  \overset{\partial}{\longrightarrow}
  \operatorname{SQS}^{\pm}(\T)
  \longrightarrow1.
\]

For each $h\in\operatorname{SQS}^{\pm}(\T)$,
fix an extension $G_h\in\Cbi$ obtained in this way.
Every $\Phi\in\Cbi$ then admits the factorization
\(
  \Phi=\Psi\circ G_{h_\Phi},
  \Psi=\Phi\circ G_{h_\Phi}^{-1}\in\ker\partial.
\)
Relative to the fixed family $(G_h)$, this factorization
is unique: the boundary map determines $h=h_\Phi$,
and then $\Psi=\Phi\circ G_h^{-1}$.

The chosen extensions, however, need not satisfy
\(
  G_{h_1\circ h_2}=G_{h_1}\circ G_{h_2}.
\)
Consequently, this construction alone does not yield
a semidirect product decomposition.
The bi-Lipschitz subclass studied in
Section~\ref{sec:regularity} admits an explicit
homomorphic section.
Related extension constructions compatible with composition were studied by
Ibragimov~\cite{Ibragimov2010}, who obtained a homomorphic extension
from quasisymmetric mappings of the real line to quasi-isometries of the
upper half plane.
\end{remark}

\section{Canonical factorization and regularity of the kernel}
\label{sec:regularity}

In this section, we refine the boundary characterization in
Theorem~\ref{thm:continuous-bidirectional-boundary} by imposing
comparability of boundary defects. Within $\Cbi$, this condition
is equivalent to bi-Lipschitz regularity of the boundary map.
We characterize the kernel of the boundary homomorphism and
describe the associated canonical factorization. This provides
the framework for uniform boundary expansions of first order in
Section~\ref{subsec:boundary-first-order} and for quasiconformal
mappings in Section~\ref{subsec:qc-kernel}.

Recall that $\delta(z)=1-|z|$, and consider the class
\[
  \Pc=\left\{
    \Phi\in\Cbi:
    1-|\Phi(z)|^2\asymp1-|z|^2
    \text{ for }z\in\D
  \right\}.
\]
Here the comparison constants may depend on $\Phi$
but are independent of $z$.
Since $1\leq1+|z|<2$, the defining estimate is equivalent to
\(
  \delta(\Phi(z))\asymp\delta(z).
\)
The following lemma establishes the equivalence between
boundary defect comparability and bi-Lipschitz boundary
regularity for arbitrary $\rho$-uniform homeomorphisms of $\D$.
\begin{samepage}
\begin{lem}
\label{lem:defect-boundary-lipschitz}
Let $F$ be a $\rho$-uniform homeomorphism of $\D$,
and let $h_F$ be its boundary map. Then
\(
  \delta(F(z))\asymp\delta(z)\) if and only if
  \(h_F\in\operatorname{BiLip}(\T).
\)
\end{lem}
\end{samepage}

\begin{proof}
Suppose first that there exist constants $a,b>0$ such that
\(
  a\delta(z)\leq\delta(F(z))\leq b\delta(z)\) for \(z\in\D\).
By Lemma~\ref{lem:uniform-boundary-rigidity},
$F$ is a quasi-isometry for $d_\D$.
Consequently, a uniform bound on the hyperbolic distance
between two points gives a uniform bound on the
hyperbolic distance between their images.

Fix $\zeta\in\T$ and $0<t<1/2$, and set
\(
  z_n=(1-2^{-n}t)\zeta\) for \(n\geq0\).
The hyperbolic distances $d_\D(z_n,z_{n+1})$ are uniformly
bounded. Hence there exists $M<1$, independent of
$n$, $t$, and $\zeta$, such that
\(
  \rho(F(z_n),F(z_{n+1}))\leq M.
\)
For any $u,v\in\D$ with $\rho(u,v)\leq M$, the identity
\(
  1-\overline u v
  =1-|u|^2+\overline u(u-v)
\)
gives
\[
  |u-v|
  \leq \frac{M(1-|u|^2)}{1-M|u|}
  \leq \frac{2M}{1-M}\,\delta(u).
\]
Applying this estimate to $F(z_n)$ and $F(z_{n+1})$,
we obtain
\[
  |F(z_n)-F(z_{n+1})|
  \leq \frac{2M}{1-M}\,\delta(F(z_n))
  \leq \frac{2Mb}{1-M}\,2^{-n}t.
\]
Summing these estimates and using
$F(z_n)\to h_F(\zeta)$ yields
\begin{equation}
\label{eq:first-order-shadow-continuous}
  |F((1-t)\zeta)-h_F(\zeta)|\leq Ct,
  \qquad \zeta\in\T,\quad 0<t<1/2,
\end{equation}
where $C$ is independent of $t$ and $\zeta$.

Now let $\zeta,\eta\in\T$ satisfy
\(
  0<s:=d_\T(\zeta,\eta)<1/2,
\)
and put $z=(1-s)\zeta$ and $w=(1-s)\eta$.
The formula for $\rho$ in polar coordinates gives a uniform
upper bound on $d_\D(z,w)$.
The quasi-isometry property of $F$ and the preceding
Euclidean estimate therefore imply
\(
  |F(z)-F(w)|\leq C_1s,
\)
where $C_1$ is independent of $\zeta$ and $\eta$.
Together with \eqref{eq:first-order-shadow-continuous},
this gives
\[
  |h_F(\zeta)-h_F(\eta)|
  \leq |h_F(\zeta)-F(z)|
       +|F(z)-F(w)|
       +|F(w)-h_F(\eta)|
  \leq (2C+C_1)s.
\]
The estimate extends to antipodal pairs by continuity.
Since the chordal metric and $d_\T$ are bi-Lipschitz
equivalent, $h_F$ is Lipschitz with respect to $d_\T$.

The inverse map satisfies
\(
  b^{-1}\delta(w)
  \leq\delta(F^{-1}(w))
  \leq a^{-1}\delta(w)\) for \(w\in\D\).
Applying the same argument to $F^{-1}$ shows that
$h_{F^{-1}}=h_F^{-1}$ is Lipschitz.
Thus $h_F\in\operatorname{BiLip}(\T)$.

Conversely, suppose that $h=h_F$ is bi-Lipschitz.
Define its radial extension by
\(
  E_h(0)=0,
   E_h(r\zeta)=rh(\zeta)\) for \(0<r<1,\ \zeta\in\T\).
The map $E_h$ is bi-Lipschitz with respect to $\rho$
and preserves modulus; see
\cite[proof of Theorem~4.8]{Wu2026}.
Its inverse is $E_{h^{-1}}$.
Consequently,
\(
  \Theta=F\circ E_{h^{-1}}
\)
is a $\rho$-uniform homeomorphism of $\D$ whose
boundary map is
\(
  h_\Theta=h\circ h^{-1}=\operatorname{id}_\T.
\)
By Lemma~\ref{lem:uniform-boundary-rigidity},
\(
  D_\Theta:=
  \sup_{z\in\D}d_\D(\Theta(z),z)<\infty.
\)
Hence
\(
  \sup_{z\in\D}\rho(\Theta(z),z)
  \leq\tanh(D_\Theta/2)<1.
\)
Lemma~\ref{lem:bounded-displacement-preserver}
therefore gives
\(
  \delta(\Theta(z))\asymp\delta(z).
\)
Since $F=\Theta\circ E_h$ and $E_h$ preserves modulus,
we conclude that
\(
  \delta(F(z))
  =\delta(\Theta(E_h(z)))
  \asymp\delta(E_h(z))
  =\delta(z).
\)
\end{proof}

\begin{samepage}
\begin{thm}
\label{thm:weighted-continuous-characterization}
For a continuous map $\Phi:\D\to\D$, the following conditions are equivalent:
\begin{enumerate}
\renewcommand{\labelenumi}{\textup{(\roman{enumi})}}
\item $\Phi\in\Pc$.
\item $\Phi$ is a $\rho$-uniform homeomorphism of $\D$, and $h_\Phi\in\operatorname{BiLip}(\T)$.
\item $\Phi$ is a $\rho$-uniform homeomorphism of $\D$, and $\delta(\Phi(z))\asymp\delta(z)$.
\end{enumerate}
\end{thm}
\end{samepage}

\begin{proof}
Assume~(i). By the definition of $\Pc$, we have
$\Phi\in\Cbi$ and
$\delta(\Phi(z))\asymp\delta(z)$.
Lemma~\ref{lem:biuniform-from-preservation} shows that
$\Phi$ is a $\rho$-uniform homeomorphism of $\D$.
Lemma~\ref{lem:defect-boundary-lipschitz} then gives
$h_\Phi\in\operatorname{BiLip}(\T)$, proving~(ii).

The equivalence of~(ii) and~(iii) follows directly from
Lemma~\ref{lem:defect-boundary-lipschitz}.

Finally, assume~(ii). Since
\(
  \operatorname{BiLip}(\T)
  \subset\operatorname{SQS}^{\pm}(\T),
\)
Theorem~\ref{thm:continuous-bidirectional-boundary}
gives $\Phi\in\Cbi$.
Moreover, Lemma~\ref{lem:defect-boundary-lipschitz}
yields $\delta(\Phi(z))\asymp\delta(z)$, equivalently
\(
  1-|\Phi(z)|^2\asymp1-|z|^2.
\)
Thus $\Phi\in\Pc$, proving~(i).
\end{proof}

Let $\Kc$ denote the kernel of the boundary homomorphism
restricted to $\Pc$:
\(
  \Kc=\{\Psi\in\Pc:h_\Psi=\operatorname{id}_{\T}\}.
\)
For a homeomorphism $\Psi$ of $\D$ and \(0<t<1\), define
\begin{align*}
  \omega_\Psi(t)
  &=\sup\bigl\{
    \rho(\Psi(z),\Psi(w)):
    z,w\in\D,\ \rho(z,w)\leq t
  \bigr\},\\
  \ell_\Psi(t)
  &=\inf\bigl\{
    \rho(\Psi(z),\Psi(w)):
    z,w\in\D,\ \rho(z,w)\geq t
  \bigr\}.
\end{align*}
Lemma~\ref{lem:wu-kernel-criterion} applies to arbitrary
maps of $\D$. For disk homeomorphisms, conditions
\textup{(K0)} and \textup{(K1)} are automatic.

\begin{samepage}
\begin{thm}
\label{thm:explicit-continuous-kernel}
For a map $\Psi:\D\to\D$, the following conditions
are equivalent:
\begin{enumerate}
\renewcommand{\labelenumi}{\textup{(\roman{enumi})}}
\item $\Psi\in\Kc$.
\item $\Psi$ is a $\rho$-uniform homeomorphism of $\D$
      with identity boundary values.
\item $\Psi$ is a homeomorphism of $\D$ and
\begin{equation}
\label{eq:kernel-modulus-test}
  \sup_{z\in\D}\rho(\Psi(z),z)<1,
  \qquad
  \lim_{t\downarrow0}\omega_\Psi(t)=0,
  \qquad
  \ell_\Psi(t)>0
  \quad (0<t<1).
\end{equation}
\end{enumerate}
\end{thm}
\end{samepage}

\begin{proof}
We first note that, for a homeomorphism $\Psi$,
the condition $\omega_\Psi(t)\to0$ as $t\downarrow0$
is equivalent to uniform continuity of $\Psi$
with respect to $\rho$.
Likewise, positivity of $\ell_\Psi(t)$ for every
$0<t<1$ is equivalent to uniform continuity of $\Psi^{-1}$.
Indeed, $\ell_\Psi(t)>0$ gives
\(
  \rho(\Psi(z),\Psi(w))<\ell_\Psi(t)
  \Longrightarrow
  \rho(z,w)<t.
\)
Conversely, uniform continuity of $\Psi^{-1}$ gives,
for each $0<t<1$, a number $\eta_t>0$ such that
\(
  \rho(\Psi(z),\Psi(w))<\eta_t
  \Longrightarrow
  \rho(z,w)<t.
\)
Taking the contrapositive yields
$\ell_\Psi(t)\geq\eta_t>0$.

Assume~(i).
By the definition of $\Kc$, we have
$\Psi\in\Pc$ and $h_\Psi=\operatorname{id}_\T$.
Theorem~\ref{thm:weighted-continuous-characterization}
shows that $\Psi$ is a $\rho$-uniform homeomorphism
of $\D$, proving~(ii).

Assume~(ii).
Lemma~\ref{lem:uniform-boundary-rigidity} gives
\(
  D_\Psi:=
  \sup_{z\in\D}d_\D(\Psi(z),z)<\infty.
\)
Since $\rho=\tanh(d_\D/2)$, it follows that
\(
  \sup_{z\in\D}\rho(\Psi(z),z)
  \leq\tanh(D_\Psi/2)<1.
\)
The other two conditions in
\eqref{eq:kernel-modulus-test} follow from the
uniform continuity of $\Psi$ and $\Psi^{-1}$,
as explained above.
Thus~(iii) holds.

Finally, assume~(iii).
The two modulus conditions imply that $\Psi$
is a $\rho$-uniform homeomorphism.
Lemma~\ref{lem:bounded-displacement-preserver}
therefore gives $\Psi\in\Cbi$ and
$\delta(\Psi(z))\asymp\delta(z)$.
Hence $\Psi\in\Pc$.

It remains to identify its boundary map.
Set
\(
  M:=\sup_{z\in\D}\rho(\Psi(z),z)<1.
\)
Using
\(
  1-\overline z\Psi(z)
  =1-|z|^2+\overline z\bigl(z-\Psi(z)\bigr),
\)
we obtain
\(
  |\Psi(z)-z|
  \leq M\bigl(1-|z|^2+|z|\,|\Psi(z)-z|\bigr).
\)
Consequently,
\[
  |\Psi(z)-z|
  \leq\frac{M(1-|z|^2)}{1-M|z|}
  \leq\frac{2M}{1-M}\,\delta(z).
\]
For every $\zeta\in\T$ and $0<r<1$, this yields
\[
  |\Psi(r\zeta)-\zeta|
  \leq
  \left(1+\frac{2M}{1-M}\right)(1-r)
  \longrightarrow0
  \qquad (r\uparrow1).
\]
The continuous boundary extension supplied by
Lemma~\ref{lem:uniform-boundary-rigidity}
therefore satisfies $h_\Psi=\operatorname{id}_\T$.
Thus $\Psi\in\Kc$, proving~(i).
\end{proof}

\begin{remark}
The displacement bound in \eqref{eq:kernel-modulus-test}
does not by itself imply uniform continuity of $\Psi$
or $\Psi^{-1}$ with respect to $\rho$.
Under the hypotheses of
Theorem~\ref{thm:uniform-boundary-normal}, positivity of
the boundary normal coefficient ensures both uniform
continuity conditions.
\end{remark}

We next describe the group structure associated with the
canonical factorization.

\begin{samepage}
\begin{cor}
\label{cor:continuous-semidirect-product}
Every $\Phi\in\Pc$ admits a unique factorization
\(
 \Phi=\Psi\circ E_h,
 (\Psi,h)\in\Kc\times\operatorname{BiLip}(\T).
\)
The factors are given by
\(
 h=h_\Phi,
 \Psi=\Phi\circ E_{h^{-1}}.
\)
Consequently, the map
\[
 \Xi:\Kc\rtimes_\alpha\operatorname{BiLip}(\T)
 \longrightarrow\Pc,
 \qquad
 \Xi(\Psi,h)=\Psi\circ E_h,
\]
is an isomorphism of groups, where
\(
 \alpha_h(\Psi)=E_h\circ\Psi\circ E_{h^{-1}}.
\)
\end{cor}
\end{samepage}

\begin{proof}
The factorization and semidirect product description are the continuous
restrictions of \cite[Theorem~5.2 and Corollary~5.3]{Wu2026},
where the general preserver class and its kernel are monoids.
In the continuous setting they are groups, as the following direct
argument verifies.

By Theorem~\ref{thm:weighted-continuous-characterization},
$\Pc$ is a group under composition: both $\rho$-uniformity
and bi-Lipschitz regularity of the boundary map are preserved
under composition and inversion.
Since continuous boundary extensions respect composition,
the map
\(
 \partial:\Pc\longrightarrow\operatorname{BiLip}(\T),
 \partial\Phi=h_\Phi,
\)
is a group homomorphism with kernel $\Kc$.
In particular, $\Kc$ is a normal subgroup of $\Pc$.

For every $h\in\operatorname{BiLip}(\T)$, its radial extension
$E_h$ is a $\rho$-bi-Lipschitz homeomorphism with boundary map $h$,
as recalled in the proof of
Lemma~\ref{lem:defect-boundary-lipschitz}.
Hence $E_h\in\Pc$ by
Theorem~\ref{thm:weighted-continuous-characterization}.
Moreover,
\(
 E_{h_1}\circ E_{h_2}=E_{h_1\circ h_2},
 E_h^{-1}=E_{h^{-1}},
 h_{E_h}=h.
\)
Thus $h\mapsto E_h$ is a homomorphic section of $\partial$.

Given $\Phi\in\Pc$, set $h=h_\Phi$ and
$\Psi=\Phi\circ E_{h^{-1}}$.
Then $\Psi\in\Pc$ and
\(
 h_\Psi=h_\Phi\circ h^{-1}=\operatorname{id}_{\T},
\)
so $\Psi\in\Kc$ and $\Phi=\Psi\circ E_h$.
Conversely, if $\Phi=\widetilde\Psi\circ E_{\widetilde h}$
with $\widetilde\Psi\in\Kc$ and
$\widetilde h\in\operatorname{BiLip}(\T)$,
taking boundary values gives $\widetilde h=h_\Phi=h$.
It follows that
$\widetilde\Psi=\Phi\circ E_{h^{-1}}=\Psi$,
proving uniqueness.

Normality of $\Kc$ and the identities for radial extensions
show that $h\mapsto\alpha_h$ defines an action of
$\operatorname{BiLip}(\T)$ on $\Kc$ by group automorphisms.
With the multiplication
\(
 (\Psi_1,h_1)(\Psi_2,h_2)
 =
 \bigl(\Psi_1\circ\alpha_{h_1}(\Psi_2),\,
       h_1\circ h_2\bigr),
\)
we have
\[
 (\Psi_1\circ E_{h_1})\circ(\Psi_2\circ E_{h_2})
 =
 \bigl(\Psi_1\circ\alpha_{h_1}(\Psi_2)\bigr)
 \circ E_{h_1\circ h_2}.
\]
Therefore $\Xi$ is a group homomorphism.
Its bijectivity follows from the existence and uniqueness
of the factorization.
\end{proof}

\begin{example}
\label{ex:power-stretch}
Let $\mathbb H=\{w\in\mathbb C:\operatorname{Im}w>0\}$.
For $\alpha>0$, $\alpha\ne1$, define
\(
 F_\alpha(w)=w|w|^{\alpha-1},
 \Phi_\alpha=\tau^{-1}\circ F_\alpha\circ\tau,
 \tau(z)=i\frac{1+z}{1-z}.
\)
In polar coordinates,
\[
 F_\alpha(re^{i\theta})=r^\alpha e^{i\theta},
 \qquad r>0,\quad 0<\theta<\pi.
\]
Thus $F_\alpha$ is a homeomorphism of $\mathbb H$ with
inverse $F_{1/\alpha}$. Its radial and tangential stretches are
$\alpha r^{\alpha-1}$ and $r^{\alpha-1}$, respectively.
Consequently, $F_\alpha$ is quasiconformal, and conformal
conjugation gives
\(
 K(\Phi_\alpha)=K(F_\alpha)
 =\max\{\alpha,\alpha^{-1}\}.
\)

The boundary map of $F_\alpha$ is
\(
 h_\alpha(x)=\operatorname{sgn}(x)|x|^\alpha.
\)
It is locally absolutely continuous, with
\(
 h_\alpha'(x)=\alpha|x|^{\alpha-1}\) for almost every \(x\in\mathbb R\).
For every $p>\max\{1,\alpha\}$, one has
\(
 -1<\alpha-1<p-1.
\)
By the standard criterion for power weights \cite{Grafakos2014},
$h_\alpha'\in A_p(\mathbb R)\subset A_\infty(\mathbb R)$.
The same argument, with $\alpha$ replaced by $\alpha^{-1}$,
applies to $h_\alpha^{-1}=h_{1/\alpha}$.
Hence both $h_\alpha$ and its inverse are strongly quasisymmetric.

Applying the characterization of Gonz\'alez and Nicolau
\cite[p.~284]{GN1998} separately to $F_\alpha$ and
$F_\alpha^{-1}=F_{1/\alpha}$ shows that $F_\alpha$ preserves
$H^\infty(\mathbb H)$ interpolating sequences in both directions.
Since the interpolation property is invariant under conformal
equivalence, it follows that $\Phi_\alpha\in\Cbi$.

On the other hand, for $y>0$, set
\(
 z_y=\tau^{-1}(iy)=\frac{y-1}{y+1}.
\)
Since $F_\alpha(iy)=iy^\alpha$, we obtain
\(
 \Phi_\alpha(z_y)=\frac{y^\alpha-1}{y^\alpha+1}.
\)
Therefore
\[
 1-|z_y|^2=\frac{4y}{(1+y)^2},
 \qquad
 1-|\Phi_\alpha(z_y)|^2
 =\frac{4y^\alpha}{(1+y^\alpha)^2},
\]
and hence
\(
 q_{\Phi_\alpha}(z_y)
 =y^{\alpha-1}\frac{(1+y)^2}{(1+y^\alpha)^2}.
\)
As $y\downarrow0$, this tends to zero if $\alpha>1$
and to infinity if $0<\alpha<1$.
Thus the boundary defects are not uniformly comparable,
so $\Phi_\alpha\notin\Pc$.
Consequently, $\Pc\subsetneq\Cbi$, even within the
quasiconformal class.
\end{example}

\subsection{Uniform boundary expansions of first order}
\label{subsec:boundary-first-order}

We now characterize the canonical kernels that admit a uniform
radial expansion of first order at the boundary.
No differentiability in the interior is assumed.
Throughout this subsection, differentials are understood in the
real sense, with $\mathbb C$ identified with $\mathbb R^2$.

\begin{definition}
\label{def:uniform-boundary-expansion}
Let $\Psi:\overline\D\to\overline\D$ be continuous with identity
boundary values. We say that $\Psi$ has a uniform boundary expansion of first
order if there is a continuous map
$v_\Psi:\T\to\mathbb C$ such that
\begin{equation}\label{eq:uniform-boundary-expansion}
 \lim_{t\downarrow0}\sup_{\zeta\in\T}
 \frac{\left|\Psi((1-t)\zeta)-\zeta+t v_\Psi(\zeta)\right|}{t}=0.
\end{equation}
The map $v_\Psi$ is uniquely determined by $\Psi$. We write
\(
 a_\Psi(\zeta)
 :=\operatorname{Re}\!\left(\overline\zeta\,v_\Psi(\zeta)\right).
\)
\end{definition}

Condition~\eqref{eq:uniform-boundary-expansion} controls both the normal
and tangential components of the displacement to first order. The coefficient
$a_\Psi(\zeta)$ is the derivative at $t=0^+$ of
$t\mapsto\delta(\Psi((1-t)\zeta))$; it measures the change of the boundary
defect when approaching the boundary from inside the disk.

\begin{thm}
\label{thm:uniform-boundary-normal}
Let $\Psi:\overline\D\to\overline\D$ be a homeomorphism with identity
boundary values and a uniform boundary expansion of first order.
Then $\Psi$ is differentiable at every $\zeta\in\T$ relative to
$\overline\D$, with real linear differential
\begin{equation}\label{eq:boundary-differential}
 D_\partial\Psi(\zeta)[h]
 =h+\operatorname{Re}(\overline\zeta h)
       \bigl(v_\Psi(\zeta)-\zeta\bigr),
 \qquad h\in\mathbb C.
\end{equation}
The following conditions are equivalent:
\begin{enumerate}
\renewcommand{\labelenumi}{\textup{(\roman{enumi})}}
\item $\Psi|_{\D}\in\Kc$.
\item $\displaystyle\inf_{\zeta\in\T}a_\Psi(\zeta)>0$.
\item $\det D_\partial\Psi(\zeta)>0$ for every $\zeta\in\T$.
\end{enumerate}
Moreover,
\begin{equation}\label{eq:c1-normal-limit}
 \lim_{r\uparrow1}\sup_{\zeta\in\T}
 \left|
 \frac{1-|\Psi(r\zeta)|^2}{1-r^2}-a_\Psi(\zeta)
 \right|=0.
\end{equation}
\end{thm}

\begin{proof}
Set $B(\zeta)=\zeta-v_\Psi(\zeta)$. With $t=1-r$, the assumed
expansion becomes
\begin{equation}\label{eq:boundary-expansion-displacement}
 \Psi(r\zeta)=r\zeta+tB(\zeta)+o(t),
\end{equation}
uniformly in $\zeta\in\T$. It also gives
\(
 1-|\Psi(r\zeta)|^2=2t\,a_\Psi(\zeta)+o(t),
 1-r^2=2t-t^2.
\)
This proves \eqref{eq:c1-normal-limit} and, equivalently,
\begin{equation}\label{eq:boundary-expansion-defect}
 \delta(\Psi(r\zeta))
 =t\bigl(a_\Psi(\zeta)+o(1)\bigr)
\end{equation}
uniformly in $\zeta$.

We next verify the boundary differentiability assertion. Fix
$\zeta_0\in\T$ and write $z=(1-t)\eta$, $h=z-\zeta_0$, and
$x=\operatorname{Re}(\overline{\zeta_0}h)$. As $z\to\zeta_0$ within
$\D$, one has $t\leq|h|$, $\eta\to\zeta_0$, and
\(
 t+x
 =-(1-t)\bigl(1-\operatorname{Re}(\overline{\zeta_0}\eta)\bigr)
 =O(|h|^2).
\)
Let $L_{\zeta_0}$ be the real linear map on the right side of
\eqref{eq:boundary-differential}. Equation~\eqref{eq:boundary-expansion-displacement}
and continuity of $B$ give
\[
 \Psi(z)-\Psi(\zeta_0)-L_{\zeta_0}[h]
 =t\bigl(B(\eta)-B(\zeta_0)\bigr)
   +(t+x)B(\zeta_0)+o(t)
 =o(|h|).
\]
For $z\in\T$, the same conclusion follows from the identity boundary
values and $x=O(|h|^2)$. Thus the differential in
\eqref{eq:boundary-differential} satisfies
\[
 \lim_{\substack{z\to\zeta_0\\z\in\overline\D}}
 \frac{|\Psi(z)-\Psi(\zeta_0)
       -D_\partial\Psi(\zeta_0)[z-\zeta_0]|}{|z-\zeta_0|}=0.
\]
Writing $v_\Psi(\zeta)=a_\Psi(\zeta)\zeta+b_\Psi(\zeta)i\zeta$,
the matrix of $D_\partial\Psi(\zeta)$ in the positively oriented
orthonormal basis $(\zeta,i\zeta)$ is
\(
 \begin{pmatrix}
 a_\Psi(\zeta)&0\\ b_\Psi(\zeta)&1
 \end{pmatrix}.
\)
Hence $\det D_\partial\Psi(\zeta)=a_\Psi(\zeta)$, and continuity on
the compact circle proves $(ii)\Leftrightarrow(iii)$. If (i) holds,
the lower bound for the boundary defect ratio and
\eqref{eq:c1-normal-limit} give (ii).

It remains to prove $(ii)\Rightarrow(i)$. Suppose that
$\inf_{\T}a_\Psi>0$. Equation~\eqref{eq:boundary-expansion-defect}
gives $\delta(\Psi(z))\asymp\delta(z)$ near the boundary. We shall
prove uniform continuity of $\Psi$ and $\Psi^{-1}$ with respect to
$\rho$ by checking sequences of asymptotically small pairs.
For $u,w\in\D$,
\(
 \frac{|u-w|}{2\delta(u)+|u-w|}
 \leq\rho(u,w)\leq\frac{|u-w|}{\delta(u)}.
\)
Consequently,
\begin{equation}\label{eq:boundary-small-pair-test}
 \rho(u_n,w_n)\longrightarrow0
 \quad\Longleftrightarrow\quad
 |u_n-w_n|=o(\delta(u_n)).
\end{equation}
Either condition also implies
$\frac{\delta(w_n)}{\delta(u_n)}\to1$.

First suppose that $\rho(z_n,w_n)\to0$. Every subsequence has a
further subsequence on which $z_n\to\xi\in\overline\D$, and then
$w_n\to\xi$ as well. If $\xi\in\D$, continuity gives
$\rho(\Psi(z_n),\Psi(w_n))\to0$. If $\xi\in\T$, write
\(
 z_n=(1-t_n)\zeta_n,
 w_n=(1-s_n)\eta_n.
\)
Equation~\eqref{eq:boundary-small-pair-test} yields
$|z_n-w_n|=o(t_n)$ and $s_n/t_n\to1$, while
$\zeta_n,\eta_n\to\xi$. Continuity of $B$ implies
\(
 t_nB(\zeta_n)-s_nB(\eta_n)=o(t_n).
\)
Subtracting \eqref{eq:boundary-expansion-displacement} at the two
points gives $|\Psi(z_n)-\Psi(w_n)|=o(t_n)$. Since
$\delta(\Psi(z_n))\asymp t_n$,
\eqref{eq:boundary-small-pair-test} again gives
$\rho(\Psi(z_n),\Psi(w_n))\to0$. This proves uniform continuity
of $\Psi$ with respect to $\rho$.

For the inverse, suppose that
$\rho(\Psi(z_n),\Psi(w_n))\to0$. By compactness, every subsequence
has a further subsequence on which
$\Psi(z_n)\to\xi\in\overline\D$, and then
$\Psi(w_n)\to\xi$ as well. If $\xi\in\D$, continuity of the
inverse gives $\rho(z_n,w_n)\to0$. If $\xi\in\T$, continuity of
the inverse on the closed disk and its identity boundary values imply
$z_n,w_n\to\xi$. Use the same radial notation as above. The criterion for
pairs of nearby points gives
\(
 |\Psi(z_n)-\Psi(w_n)|=o(t_n),
 \frac{\delta(\Psi(w_n))}{\delta(\Psi(z_n))}\rightarrow1.
\)
By \eqref{eq:boundary-expansion-defect}, the latter ratio equals
\(
 \frac{s_n}{t_n}
 \frac{a_\Psi(\eta_n)+o(1)}{a_\Psi(\zeta_n)+o(1)}.
\)
Since $\zeta_n,\eta_n\to\xi$ and $a_\Psi(\xi)>0$, we obtain
$s_n/t_n\to1$. Continuity of $B$ and subtraction in
\eqref{eq:boundary-expansion-displacement} now yield
$|z_n-w_n|=o(t_n)$. Thus $\rho(z_n,w_n)\to0$, proving uniform
continuity of $\Psi^{-1}$. Theorem~\ref{thm:explicit-continuous-kernel}
therefore gives $\Psi|_{\D}\in\Kc$.
\end{proof}

\begin{remark}
Let $\Psi:\overline\D\to\overline\D$ be continuous, with
$\Psi(\D)\subseteq\D$ and
$\Psi|_{\T}=\operatorname{id}_{\T}$, and suppose that $\Psi$
has a uniform boundary expansion of first order. Then
\[
 \Psi|_{\D}\in\Kc
 \quad\Longleftrightarrow\quad
 \Psi|_{\D}\text{ is injective and }
 \inf_{\zeta\in\T}a_\Psi(\zeta)>0.
\]

Indeed, the prescribed boundary values make $\Psi|_{\D}$ proper.
For every compact set $K\subset\D$, the preimage $\Psi^{-1}(K)$
is compact in $\overline\D$ and disjoint from $\T$, hence compact
in $\D$.
If $\Psi|_{\D}$ is injective, invariance of domain makes its image
open in $\D$, while properness makes its image closed in $\D$.
Since the image is nonempty and $\D$ is connected, it equals $\D$.
Together with the identity boundary values, this shows that
$\Psi$ is a continuous bijection of $\overline\D$ onto itself.
By compactness, $\Psi$ is therefore a homeomorphism.

Thus, under the stated assumptions, injectivity of $\Psi|_{\D}$
is equivalent to $\Psi$ being a homeomorphism of $\overline\D$.
The asserted equivalence now follows from
Theorem~\ref{thm:uniform-boundary-normal}, since every element
of $\Kc$ is a homeomorphism of $\D$.
The boundary condition does not replace interior injectivity.
\end{remark}

\begin{cor}
\label{cor:c1-boundary-normal}
Let $\Psi:\overline\D\to\overline\D$ be a homeomorphism of class
$C^1(\overline\D)$ with identity boundary values, where
$C^1(\overline\D)$ means that $\Psi$ has a $C^1$ extension to a
neighborhood of $\overline\D$. Define
\(
 a_\Psi(\zeta)
 :=\operatorname{Re}\!\left(
 \overline\zeta\,D\Psi(\zeta)[\zeta]
 \right)\) for \(\zeta\in\T\).
Then the following conditions are equivalent:
\begin{enumerate}
\renewcommand{\labelenumi}{\textup{(\roman{enumi})}}
\item $\Psi|_{\D}\in\Kc$.
\item $\displaystyle\inf_{\zeta\in\T}a_\Psi(\zeta)>0$.
\item $\det D\Psi(\zeta)>0$ for every $\zeta\in\T$.
\end{enumerate}
\end{cor}

\begin{proof}
Uniform Taylor expansion of first order on the compact boundary
gives \eqref{eq:uniform-boundary-expansion} with
$v_\Psi(\zeta)=D\Psi(\zeta)[\zeta]$.
The relative boundary differential agrees with $D\Psi$ on $\T$,
so all assertions follow from
Theorem~\ref{thm:uniform-boundary-normal}.
\end{proof}

\begin{example}
\label{ex:pointwise-boundary-insufficient}
There is a homeomorphism $\Psi:\overline\D\to\overline\D$ with identity
boundary values whose restriction to $\D$ does not belong to $\Kc$,
although $\Psi$ admits a homeomorphic extension
$\widehat\Psi:\mathbb C\to\mathbb C$ that is differentiable at every
$\zeta\in\T$, with $D\widehat\Psi(\zeta)=I$.

For $n\geq3$, set
\(
 \theta_n=2^{-n},
 \delta_n=2^{-3n},
 c_n=(1-\delta_n)e^{i\theta_n},
 R_n=\delta_n/4.
\)
The closed disks $\overline B(c_n,R_n)$ lie in $\D$, are pairwise
disjoint, and accumulate only at $1$. Indeed,
$|c_n|+R_n=1-3\delta_n/4<1$, and, if $m>n$, then
\[
 |c_n-c_m|
 \geq(1-\delta_m)\sin(\theta_n-\theta_m)
 \geq\frac{\theta_n}{8}
 >\frac{9\delta_n}{32}
 \geq R_n+R_m.
\]

Let $h_n:[0,R_n]\to[0,R_n]$ be the increasing piecewise affine
homeomorphism whose graph passes through
\(
 (0,0),
 (R_n/n,R_n/2),
 (R_n,R_n).
\)
Define a map on $\mathbb C$ by
\[
 \widehat\Psi(c_n+su)=c_n+h_n(s)u
 \quad(0<s\leq R_n,\ |u|=1),
 \qquad
 \widehat\Psi(c_n)=c_n,
\]
and put $\widehat\Psi(z)=z$ outside these disks.

Each disk is mapped homeomorphically onto itself, with its boundary
fixed, so the resulting map is bijective. The disks are locally
finite away from $1$; hence both $\widehat\Psi$ and its inverse are
continuous there. Their displacements on the $n$th disk are at most
$2R_n\to0$, which also gives continuity at $1$.
Thus $\widehat\Psi$ is a homeomorphism of $\mathbb C$, and
$\Psi:=\widehat\Psi|_{\overline\D}$ is a homeomorphism of
$\overline\D$ with identity boundary values.

At every boundary point other than $1$, the map agrees locally
with the identity. For $z\in\overline B(c_n,R_n)$, one has
\(
 |z-1|
 \geq(1-\delta_n)\sin\theta_n-R_n
 \geq\frac{\theta_n}{8}.
\)
Consequently,
\(
 \frac{|\widehat\Psi(z)-z|}{|z-1|}
 \leq\frac{4\delta_n}{\theta_n}
 =4\cdot2^{-2n}\rightarrow0.
\)
Outside the disks the numerator vanishes. Moreover, any sequence
of points in these disks tending to $1$ must have disk indices
tending to infinity. Therefore $D\widehat\Psi(1)=I$ as well.
In particular, the boundary differential is constant and invertible,
and its normal component is identically one.

Nevertheless, set
\(
 w_n=c_n+\frac{\delta_n}{4n}e^{i\theta_n}.
\)
The definition gives
\(
 \Psi(w_n)=c_n+\frac{\delta_n}{8}e^{i\theta_n},
 \Psi(c_n)=c_n.
\)
Direct calculation yields
\[
 \rho(c_n,w_n)
 =\frac{1}{8n-1-(4n-1)\delta_n}
 \longrightarrow0,\qquad
 \rho(\Psi(c_n),\Psi(w_n))
 =\frac{1}{15-7\delta_n}
 \longrightarrow\frac1{15}.
\]
Thus $\Psi|_{\D}$ is not uniformly continuous with respect to
$\rho$. By Theorem~\ref{thm:explicit-continuous-kernel},
$\Psi|_{\D}\notin\Kc$.

The failure of uniformity can also be seen directly. Since
$D\widehat\Psi(\zeta)=I$ for every $\zeta\in\T$, any uniform boundary
expansion of first order would necessarily have
$v_\Psi(\zeta)=\zeta$. However,
\[
 \frac{|\Psi(w_n)-w_n|}{1-|w_n|}
 =\frac{n-2}{2(4n-1)}
 \longrightarrow\frac18,
\]
contradicting the required uniform $o(t)$ remainder.
Hence pointwise boundary differentiability, even with a continuous
and everywhere invertible boundary differential, cannot replace
a uniform boundary expansion of first order.
\end{example}

\begin{samepage}
\begin{cor}
\label{cor:uniform-boundary-factor-classes}
Let $\Pc^{1,\partial}$ denote the maps in $\Pc$ whose canonical kernel
admits a uniform boundary expansion of first order. As a set,
$\Pc^{1,\partial}$ consists precisely of the maps
\(
 \Phi=\Psi\circ E_h, h\in\operatorname{BiLip}(\T),
\)
where $\Psi$ is a homeomorphism of $\overline\D$ that has identity boundary
values and admits the expansion~\eqref{eq:uniform-boundary-expansion} with
\(
 \inf_{\zeta\in\T}
 \operatorname{Re}\bigl(\overline\zeta v_\Psi(\zeta)\bigr)>0.
\)
Both $h$ and $\Psi$ are uniquely determined by $\Phi$.
\end{cor}
\end{samepage}

\begin{proof}
By Theorem~\ref{thm:uniform-boundary-normal}, a homeomorphism
$\Psi$ with the stated boundary expansion satisfies
$\Psi|_{\D}\in\Kc$ if and only if the displayed positivity
condition holds. The parametrization and uniqueness therefore
follow from Corollary~\ref{cor:continuous-semidirect-product}.
\end{proof}

\begin{samepage}
\begin{cor}
\label{cor:c1-factor-classes}
Let $\Pc^{1,\mathrm{cl}}$ denote the maps in $\Pc$ whose canonical
kernels extend to maps of class $C^1(\overline\D)$.
As a set, $\Pc^{1,\mathrm{cl}}$ consists precisely of the maps
\(
 \Phi=(\Psi|_{\D})\circ E_h,
 h\in\operatorname{BiLip}(\T),
\)
where $\Psi$ is a homeomorphism of $\overline\D$ of class
$C^1(\overline\D)$ satisfying
\(
 \Psi|_{\T}=\operatorname{id}_{\T},
 \inf_{\zeta\in\T}\operatorname{Re}
 \bigl(\overline\zeta\,D\Psi(\zeta)[\zeta]\bigr)>0.
\)
Both $h$ and $\Psi$ are uniquely determined by $\Phi$.
\end{cor}
\end{samepage}

\begin{proof}
Combine Corollary~\ref{cor:c1-boundary-normal} with
Corollary~\ref{cor:continuous-semidirect-product}.
The latter also gives uniqueness.
This parametrization concerns only the underlying set; no group structure
is asserted for the subclass with the additional $C^1$ condition.
\end{proof}

\begin{example}
\label{ex:rough-interior-kernel}
Define $f:[0,1]\to[0,1]$ by
\[
 f(r)=
 \begin{cases}
  2r^2,&0\leq r\leq\frac12,\\[2pt]
  r,&\frac12\leq r\leq1,
 \end{cases}
\]
and put $\Psi(0)=0$ and
$\Psi(r\zeta)=f(r)\zeta$ for $0<r\leq1$ and $\zeta\in\T$.
The function $f$ is a strictly increasing homeomorphism of $[0,1]$,
so $\Psi$ is a homeomorphism of the closed disk.
It is the identity for $|z|\geq1/2$, hence its uniform boundary
expansion has $v_\Psi(\zeta)=\zeta$ and $a_\Psi\equiv1$.
Theorem~\ref{thm:uniform-boundary-normal} therefore gives
$\Psi|_{\D}\in\Kc$.

Since $f'_-(1/2)=2$ and $f'_+(1/2)=1$, the map $\Psi$
is not differentiable at any point of the circle $|z|=1/2$.
Taking $\Phi=\Psi|_{\D}$ and $h=\operatorname{id}_{\T}$ therefore
gives
$\Phi\in\Pc^{1,\partial}\setminus\Pc^{1,\mathrm{cl}}$.
Since every kernel of class $C^1$ on the closed disk admits a uniform boundary
expansion of first order, this proves
\(
 \Pc^{1,\mathrm{cl}}\subsetneq\Pc^{1,\partial}.
\)
\end{example}

A nonvanishing interior Jacobian is not required for membership
in $\Kc$. Indeed, $\Psi(z)=|z|^2z$ is a smooth homeomorphism
of $\overline\D$ with identity boundary values and
$a_\Psi\equiv3$. Corollary~\ref{cor:c1-boundary-normal}
therefore gives $\Psi|_{\D}\in\Kc$, although $D\Psi(0)=0$.

\begin{remark}
\label{rem:kernel-origin-regularity}
Regularity of the canonical kernel $\Psi$ does not automatically
give the same regularity for $\Phi=\Psi\circ E_h$.
Even when $\Psi=\operatorname{id}_{\D}$ and $h$ is a smooth circle
diffeomorphism, $E_h$ need not be differentiable at the origin.
More precisely, $E_h$ is smooth on $\D\setminus\{0\}$ and is
real differentiable at the origin if and only if $h$ is a rotation
or a rotation composed with complex conjugation.

Indeed, radial homogeneity gives $E_h(tz)=tE_h(z)$ for
$z\in\D$ and $0<t<1$. If $E_h$ is real differentiable at the
origin, then
\[
 DE_h(0)[z]
 =\lim_{t\downarrow0}\frac{E_h(tz)-E_h(0)}{t}
 =E_h(z),
 \qquad z\in\D.
\]
Since $|E_h(z)|=|z|$, the real linear map $DE_h(0)$ is orthogonal.
Consequently, for some $\lambda\in\T$, the boundary map has the form
$h(\zeta)=\lambda\zeta$ if $h$ preserves orientation and
$h(\zeta)=\lambda\overline\zeta$ if $h$ reverses orientation.
Conversely, these boundary maps have real linear radial extensions,
which are smooth throughout $\D$.

The next subsection considers quasiconformal regularity
of the canonical kernel and of the full mapping.
\end{remark}

\subsection{Quasiconformal kernels and radial extensions}
\label{subsec:qc-kernel}

In this subsection, quasiconformal maps are homeomorphisms of $\D$ that
preserve orientation. Statements for maps that reverse orientation follow by
composition with complex conjugation.
We consider factorizations $\Phi=\Psi\circ E_h$, where $\Psi$
is a disk homeomorphism with identity boundary values and $h$ is a circle
homeomorphism that preserves orientation.
When $h$ is bi-Lipschitz, quasiconformality of $\Psi$ is equivalent
to quasiconformality of $\Phi$.
We also consider arbitrary circle homeomorphisms $h$ that preserve orientation
and determine when a quasiconformal factor
with identity boundary values yields a quasiconformal full mapping.

\begin{samepage}
\begin{lem}
\label{lem:radial-extension-qc}
Let $h:\T\to\T$ be a homeomorphism that preserves orientation, and let
$H:\mathbb R\to\mathbb R$ be a lift satisfying
\(
 h(e^{i\theta})=e^{iH(\theta)},
 H(\theta+2\pi)=H(\theta)+2\pi.
\)
Then $E_h$ is quasiconformal if and only if $h$ is bi-Lipschitz.
In this case $H$ is locally absolutely continuous,
with $H'(\theta)>0$ almost everywhere, and
\begin{equation}\label{eq:radial-qc-dilatation}
 K(E_h)
 =\operatorname*{ess\,sup}_{\theta\in[0,2\pi]}
 \max\left\{H'(\theta),\frac{1}{H'(\theta)}\right\}.
\end{equation}
\end{lem}
\end{samepage}

\begin{proof}
The equivalence is the case of the unit circle in
\cite[Theorem~4.2]{Kalaj2014}.
By radial homogeneity, quasiconformality on the disk is equivalent to
quasiconformality on the whole plane, with the same maximal dilatation.
Under these conditions, the lift $H$ is bi-Lipschitz, hence locally
absolutely continuous with derivative bounded above and bounded
away from zero almost everywhere.
Formula~\eqref{eq:radial-qc-dilatation} follows from
\cite[Theorem~2.2, equation~(2.11)]{Kalaj2014}.
\end{proof}

\begin{prop}
\label{prop:radial-extension-differentials}
Let $h\in\operatorname{BiLip}(\T)$, and let $H$ be a lift of $h$.
Then $E_h$ is quasiconformal if $h$ preserves orientation and
antiquasiconformal if $h$ reverses orientation. Moreover,
\[
 K(E_h)
 =\operatorname*{ess\,sup}_{\theta\in[0,2\pi]}
 \max\left\{|H'(\theta)|,\frac{1}{|H'(\theta)|}\right\}.
\]
For an antiquasiconformal map $F$, we set
$K(F)=K(F\circ J)$, where $J(z)=\overline z$.
\end{prop}

\begin{proof}
When $h$ preserves orientation, the conclusion follows from
Lemma~\ref{lem:radial-extension-qc}.
If $h$ reverses orientation, set $j=J|_{\T}$ and $g=h\circ j$.
Then $g\in\operatorname{BiLip}^{+}(\T)$ and
$E_g=E_h\circ J$, so the same lemma shows that $E_h$
is antiquasiconformal and $K(E_h)=K(E_g)$.
The map $G(\theta)=H(-\theta)$ is a lift of $g$, with
$G'(\theta)=|H'(-\theta)|$ almost everywhere.
Applying the dilatation formula in that lemma and using
the $2\pi$-periodicity of $|H'|$ gives the stated formula.
\end{proof}

For a fixed quasiconformal factor with identity boundary
values, the radial extension criterion determines exactly
which boundary maps yield elements of $\Pc$.

\begin{samepage}
\begin{thm}
\label{thm:qc-kernel-classification}
Let $\Psi:\D\to\D$ be a quasiconformal homeomorphism with
identity boundary values. For any circle homeomorphism $h$ that preserves
orientation, set $\Phi=\Psi\circ E_h$.
Then the following conditions are equivalent:
\begin{enumerate}
\renewcommand{\labelenumi}{\textup{(\roman{enumi})}}
\item $\Phi\in\Pc$.
\item $h\in\operatorname{BiLip}^{+}(\T)$.
\item $\Phi$ is a quasiconformal homeomorphism of $\D$.
\end{enumerate}
\end{thm}
\end{samepage}

\begin{proof}
By Lemma~\ref{lem:qc-rho-uniform}, $\Psi$ and $\Psi^{-1}$
are uniformly continuous with respect to $\rho$.
The identity boundary values and
Theorem~\ref{thm:explicit-continuous-kernel} imply $\Psi\in\Kc$.
In particular, there are constants $0<a\leq b<\infty$ such that
\(
 a\delta(z)\leq\delta(\Psi(z))\leq b\delta(z)\) for \(z\in\D\).
Since $E_h$ preserves modulus, the same constants give, for every
choice of $h$,
\begin{equation}\label{eq:qc-kernel-defect-automatic}
 a\delta(z)\leq\delta(\Phi(z))\leq b\delta(z),
 \qquad z\in\D.
\end{equation}
The boundary map of $\Phi$ is $h$.

Condition \textup{(i)} implies \textup{(ii)} by
Theorem~\ref{thm:weighted-continuous-characterization}.
If \textup{(ii)} holds, Lemma~\ref{lem:radial-extension-qc}
shows that $E_h$ is quasiconformal.
Composition gives \textup{(iii)} and the estimate
\begin{equation}\label{eq:qc-composition-bound}
 K(\Phi)\leq K(\Psi)K(E_h).
\end{equation}
Corollary~\ref{cor:continuous-semidirect-product} gives
\textup{(i)}.

Conversely, if \textup{(iii)} holds, then
$E_h=\Psi^{-1}\circ\Phi$ is quasiconformal.
Lemma~\ref{lem:radial-extension-qc} therefore gives
\textup{(ii)}.
Under the equivalent conditions, $h_\Phi=h$ and
$\Phi\circ E_{h^{-1}}=\Psi$ identify $\Psi$ as the canonical
kernel of $\Phi$.
\end{proof}

\begin{remark}
\label{rem:qc-kernel-versus-whole-map}
\begin{enumerate}
\renewcommand{\labelenumi}{\textup{(\roman{enumi})}}

\item
The boundary defect condition in the definition of $\Pc$
does not follow from preservation of Carleson sequences in
both directions, even in the quasiconformal class; see
Example~\ref{ex:power-stretch}.
For a fixed quasiconformal kernel $\Psi$ with identity boundary
values, however, estimate~\eqref{eq:qc-kernel-defect-automatic}
holds throughout the family $\Phi=\Psi\circ E_h$, including
choices of $h$ that are not bi-Lipschitz.
Such choices yield maps outside $\Pc$ despite satisfying the
boundary defect comparison.

\item
Theorem~\ref{thm:qc-kernel-classification} extends to arbitrary
circle homeomorphisms $h$ if its condition \textup{(ii)} is replaced
by $h\in\operatorname{BiLip}(\T)$ and its condition \textup{(iii)}
requires $\Phi$ to be a quasiconformal or antiquasiconformal
homeomorphism of $\D$.
Here, writing $J(z)=\overline z$, antiquasiconformality means
that $\Phi\circ J$ is quasiconformal.
Indeed, let $j=J|_{\T}$. If $h$ reverses orientation, then
$h\circ j$ preserves orientation and
\(
 \Phi\circ J=\Psi\circ E_{h\circ j}.
\)
Since $\Pc$ is a group containing $J$ and $j$ is an isometry
of $\T$, the conclusion for maps that reverse orientation follows by applying
Theorem~\ref{thm:qc-kernel-classification} to this factorization.
\end{enumerate}
\end{remark}

\begin{samepage}
\begin{cor}
\label{cor:all-qc-boundary-defect}
Let $F:\D\to\D$ be a quasiconformal or antiquasiconformal
homeomorphism, with boundary map $h_F$.
Then the following conditions are equivalent:
\begin{enumerate}
\renewcommand{\labelenumi}{\textup{(\roman{enumi})}}
\item $F\in\Pc$.
\item $h_F\in\operatorname{BiLip}(\T)$.
\item $1-|F(z)|^2\asymp 1-|z|^2$ for $z\in\D$.
\end{enumerate}
\end{cor}
\end{samepage}

\begin{proof}
By Lemma~\ref{lem:qc-rho-uniform}, every quasiconformal
homeomorphism of $\D$ and its inverse are uniformly
continuous with respect to $\rho$.
Since complex conjugation is an isometry for $\rho$,
the same conclusion holds for antiquasiconformal
homeomorphisms.
Thus, in either case, $F$ and $F^{-1}$ are uniformly
continuous with respect to $\rho$.
The equivalence follows from
Theorem~\ref{thm:weighted-continuous-characterization}.
\end{proof}

\begin{remark}
For quasiconformal homeomorphisms of $\D$ with identity
boundary values, bounded hyperbolic displacement also follows
from \cite{VuorinenZhang2014}.
Lemma~\ref{lem:uniform-boundary-rigidity} provides the corresponding
conclusion for the larger class of $\rho$-uniform
homeomorphisms with identity boundary values considered here.
\end{remark}

\section*{Declaration of competing interest}
The author declares that there are no competing interests that could
have influenced the work reported in this paper.

\section*{Funding}
This research did not receive any specific grant from funding agencies
in the public, commercial, or nonprofit sectors.

\section*{Data availability}
No datasets were generated or analysed in this theoretical study.

\end{document}